\documentclass{amsart}

\usepackage{amssymb}
\usepackage{amsmath}
\usepackage{amsthm}
\usepackage{amsfonts}
\usepackage{a4wide}

\newtheorem{lemma}{Lemma}[section]
\newtheorem{proposition}{Proposition}[section]
\newtheorem{theorem}{Theorem}[section]
\newtheorem{corollary}{Corollary}[section]

\theoremstyle{definition}
\newtheorem{definition}{Definition}[section]
\newtheorem{remark}{Remark}[section]

\newtheorem{example}{Example}[section]

\numberwithin{equation}{section}

\begin{document}


\title[\(5\)-Class towers of cyclic quartic fields]
{\(5\)-Class towers of cyclic quartic fields \\  arising from quintic reflection}

\author[A. Azizi]{Abdelmalek Azizi$^1$}
\address{$^1$Department of Mathematics, Faculty of Sciences, Mohammed First University, 60000 Oujda, Morocco}
\email{abdelmalekazizi@yahoo.fr}

\author[Y. Kishi]{Yasuhiro Kishi$^2$}
\address{$^2$Aichi University of Education, Aichi, Japan}
\email{ykishi@auecc.aichi-edu.ac.jp}

\author[D. C. Mayer]{Daniel C. Mayer$^3$}
\address{$^3$Naglergasse 53, 8010 Graz, Austria}
\email{quantum.algebra@icloud.com}
\urladdr{http://www.algebra.at}
\thanks{Research of the third author supported by the Austrian Science Fund (FWF): P 26008-N25}

\author[M. Talbi]{Mohamed Talbi$^4$}
\address{$^4$Regional center of Education and Training, 60000 Oujda, Morocco}
\email{ksirat1971@gmail.com}

\author[Mm. Talbi]{Mohammed Talbi$^5$}
\address{$^5$Regional center of Education and Training, 60000 Oujda, Morocco}
\email{talbimm@yahoo.fr}


\date{September 06, 2019}

%
%
%
%
%
%
%
%

\begin{abstract}
Let \(\zeta_5\) be a primitive fifth root of unity
and \(d\ne 1\) be a quadratic fundamental discriminant not divisible by \(5\).
For the \(5\)-dual cyclic quartic field \({M}=\mathbb{Q}((\zeta_5-\zeta_5^{-1})\sqrt{d})\)
of the quadratic fields \({k}_1=\mathbb{Q}(\sqrt{d})\) and \({k}_2=\mathbb{Q}(\sqrt{5d})\)
in the sense of the quintic reflection theorem,
the possibilities for the isomophism type of the Galois group
\(\mathrm{G}_5^{(2)}{{M}}=\mathrm{Gal}({M}_5^{(2)}/{M})\)
of the second Hilbert \(5\)-class field \({M}_5^{(2)}\) of \({M}\) are investigated,
when the \(5\)-class group \(\mathrm{Cl}_5(M)\) is elementary bicyclic of rank two.
Usually, the maximal unramified pro-\(5\)-extension \({M}_5^{(\infty)}\) of \(M\)
coincides with \({M}_5^{(2)}\) already.
The precise length \(\ell_5{M}\) of the \(5\)-class tower of \({M}\)
is determined, when  \(\mathrm{G}_5^{(2)}{{M}}\) is of order less than or equal to \(5^5\).
Theoretical results are underpinned by the actual computation of all \(83\), respectively \(93\), cases
in the range \(0<d<10^4\), respectively \(-2\cdot 10^5<d<0\).
\end{abstract}

%
%
%
%
%
%
%
%

\maketitle

%
%
%
%
%
%
%
%

\section{Introduction}
\label{s:Intro}
\noindent
The present article arose from the desire to generalize our results
\cite{ATTDM}
for the second \(3\)-class group \(\mathrm{Gal}\left({k}_3^{(2)}/{k}\right)\)
of the bicyclic biquadratic field \({k}=\mathbb{Q}\left(\sqrt{-3},\sqrt{d}\right)\),
which is the compositum of \(3\)-dual quadratic fields \({k}_1=\mathbb{Q}(\sqrt{d})\) and \({k}_2=\mathbb{Q}(\sqrt{-3d})\)
in the cubic reflection theorem,
to the situation of the quintic reflection theorem.

The precise statement of both reflection theorems requires the concept of \textit{virtual units}.
Let \(p\) be a prime number and \(K\) be a number field with multiplicative group 
\({K}^{\times}={K}\setminus\lbrace 0\rbrace\),
maximal order \(\mathcal{O}\), unit group \(U\), fractional ideal group \(\mathcal{I}\), and \(p\)-class rank \(\varrho_p\).
The quotient \(V_p=I_p/\left({K}^{\times}\right)^p\), where
$$I_p=\left\lbrace\alpha\in {K}^{\times}\mid\alpha\mathcal{O}=\mathfrak{a}^p \text{ for some } \mathfrak{a}\in\mathcal{I}\right\rbrace,$$
is an elementary abelian \(p\)-group of rank \(\sigma_p=\varrho_p+\dim_{\mathbb{F}_p}\left(U/U^p\right)\)
and is called the \(p\)-\textit{Selmer group} of non-trivial \(p\)-\textit{virtual units}, that is,
generators of principal \(p\)th powers of ideals of \({K}\).
We refer to \(\sigma_p\) as the \(p\)-\textit{Selmer rank} of \({K}\).

%
%
%
%
%
%
%
%

\subsection{Cubic reflection theorem}
\label{ss:CubicReflection}
\noindent
It is well known that the \(3\)-Selmer ranks \(\sigma_3({k}_1)\) and \(\sigma_3({k}_2)\)
of \(3\)-dual quadratic fields \({k}_1=\mathbb{Q}(\sqrt{d})\) and \({k}_2=\mathbb{Q}(\sqrt{-3d})\) (\(d>0\) square-free)
with respect to the quadratic cyclotomic mirror field \(\mathbb{Q}(\sqrt{-3})=\mathbb{Q}(\zeta_3)\), \(\zeta_3=\exp(2\pi i/3)\),
satisfy the \textit{cubic reflection theorem}
\begin{equation}
\label{eqn:3Reflection}
\sigma_3({k}_2) = \sigma_3({k}_1)-\delta,
\end{equation}
\noindent
which is a consequence of comparing the numbers of cyclic cubic extensions of \({k}_1\) and \({k}_2\)
which are unramified outside  \(3\)
from the viewpoint of both, class field theory and Kummer theory.
The invariant \(0\le\delta\le 1\) depends on the \(3\)-virtual units of \({k}_1\) and \({k}_2\).
More precisely, we have
\begin{equation}
\label{eqn:3Primary}
\delta=
\begin{cases}
0, & \text{ if } V_3(k_2) \text{ (imaginary)} \text{ contains a \(3\)-virtual unit which is not \(3\)-primary,}\\[1mm]
1, & \text{ if } V_3(k_1) \text{ (real)} \text{ contains a \(3\)-virtual unit which is not \(3\)-primary.}
\end{cases}
\end{equation}

%
%
%
%
%
%
%
%

\subsection{Quintic reflection theorem}
\label{ss:QuinticReflection}
\noindent
If \(d\ne 1\) denotes a square-free integer prime to \(5\),
then the \(5\)-Selmer ranks \(\sigma_5({k}_1)\), \(\sigma_5({k}_2)\)
of associated quadratic fields \({k}_1=\mathbb{Q}(\sqrt{d})\), \({k}_2=\mathbb{Q}(\sqrt{5d})\)
and the \(5\)-class rank \(\varrho_5({M})\)
of their \(5\)-dual cyclic quartic field \({M}=\mathbb{Q}\left((\zeta_5-\zeta_5^{-1})\sqrt{d}\right)\), \(\zeta_5=\exp(2\pi i/5)\),
with respect to the quartic cyclotomic mirror field \({k}_0=\mathbb{Q}(\zeta_5)\)
satisfy the \textit{quintic reflection theorem}
\begin{equation}
\label{eqn:5Reflection}
\varrho_5({M})=\sigma_5({k}_1)+\sigma_5({k}_2)-\delta_1-\delta_2,
\end{equation}
\noindent
where the invariants \(0\le\delta_1,\delta_2\le 1\) depend
on the \(5\)-virtual units of \({k}_1\) and \({k}_2\)
\cite[p.\,2]{Ki}.
The formula is derived by comparing the numbers of
cyclic quintic extensions of \({k}_1\), \({k}_2\) and \({M}\)
which are unramified outside of \(5\).
The maximal real subfield of \({k}_0=\mathbb{Q}(\zeta_5)\)
is the quadratic field \({k}_0^+=\mathbb{Q}(\sqrt{5})\).

%
%
%
%
%
%
%
%

\subsection{Overview}
\label{ss:Overview}
\noindent
The layout of this article is as follows.
In \S
\ref{s:GaloisAction},
we prove that
the action of the absolute Galois group \(\mathrm{Gal}(M/\mathbb{Q})\)
on the \(5\)-class group \(\mathrm{Cl}_5(M)\)
considerably reduces the possibilities for the metabelianization \(\mathrm{G}_5^{(2)}{M}\)
of the \(5\)-class tower group \(\mathrm{G}_5^{(\infty)}{M}\) of \(M\).
In \S
\ref{s:FrobeniusExtensions},
it is shown that the six unramified cyclic quintic relative extensions \(E_i/M\), \(1\le i\le 6\),
give rise to absolute extensions \(E_i/\mathbb{Q}\)
which are either Frobenius or non-Galois.
Using class number relations for the dihedral subextensions \(E_i/k_0^+\)
of \(E_i/\mathbb{Q}\),
we determine further constraints for the second \(5\)-class group \(\mathrm{G}_5^{(2)}{M}\),
the \(5\)-class tower group \(\mathrm{G}_5^{(\infty)}{M}\),
and the length \(\ell_5\) of the \(5\)-class tower
in \S
\ref{s:Metabelianization}.
The paper concludes with tables of concrete numerical realizations
in \S
\ref{s:Tables}
which underpin all theoretical statements
and additionally reveal the statistical distribution of possible cases.

%
%
%
%
%
%
%
%
\vskip 7mm

\section{\(p\)-Principalization enforced by Galois action}
\label{s:GaloisAction}
\noindent
The generating automorphism \(\sigma\)
of a cyclic number field \({F}/\mathbb{Q}\) of degree \(d\)
with Galois group \(\mathrm{Gal}({F}/\mathbb{Q})=\langle\sigma\rangle\)
acts on the class group \(\mathrm{Cl}({F})\) of \({F}\)
and thus also on the higher \(p\)-class groups \(\mathrm{G}_p^{(n)}{{F}}\)
with \(n\in\mathbb{N}\cup\lbrace\infty\rbrace\),
for a fixed prime number \(p\).
When \(d\) and \(p\) are coprime,
a remarkable restriction of the possibilities for the
metabelian second \(p\)-class group \(\mathfrak{M}=\mathrm{G}_p^{(2)}{{F}}\)
and consequently for the transfer kernel type \(\varkappa({F})\) of \({F}\)
is due to the fact that
the trace \(T_\sigma=\sum\limits_{\substack{i=0}}^{d-1}\sigma^i\) of \(\sigma\) annihilates
the commutator quotient of all the groups \(\mathrm{G}_p^{(n)}{{F}}\).

%
%
%
%

\begin{definition}
\label{dfn:AnnTrace}
Let \(p\) be a prime number and
\(G\) be a pro-\(p\)-group with finite abelianization \(G/G^\prime\).
Suppose that \(d\ge 2\) is a fixed integer.
\(G\) is said to be a \(\sigma\)-\textit{group of degree} \(d\),
if \(G\) possesses an automorphism \(\sigma\) of order \(d\)
whose trace
$$
T_\sigma
=
\sum\limits_{\substack{j=0}}^{d-1}\,\sigma^{j}
\in\mathbb{Z}\lbrack\mathrm{Aut}(G)\rbrack
$$
annihilates \(G\) modulo \(G^\prime\),
that is, if there exists \(\sigma\in\mathrm{Aut}(G)\) such that \(\mathrm{ord}(\sigma)=d\) and
$$x^{T_\sigma}=\prod\limits_{\substack{j=0}}^{d-1}\,\sigma^{j}(x)\in G^\prime $$
for all \(x\in G\).
\end{definition}

%
%
%
%

We show that an epimorphism with characteristic kernel
preserves the property of being a \(\sigma\)-group of degree \(d\).

%
%
%
%

\begin{theorem}
\label{thm:Epimorphism}
Let \(\phi:\,G\to H\) be an epimorphism of groups,
whose kernel \(\ker(\phi)\) is characteristic in \(G\).
If \(G\) is a \(\sigma\)-group of degree \(d\) coprime to \(p\),
then \(H\) is also a \(\sigma\)-group of degree \(d\).
\end{theorem}

\begin{proof}
If \(G\) is a \(\sigma\)-group of degree \(d\),
then there exists an automorphism \(\sigma\in\mathrm{Aut}(G)\) of order \(\mathrm{ord}(\sigma)=d\)
such that 
$$x^{T_\sigma}=\prod\limits_{\substack{i=0}}^{d-1}\,\sigma^i(x)\in G^\prime$$
 for all \(x\in G.\)
According to
\cite[Th. 6.2]{Ma9},
there exists an induced automorphism \(\hat{\sigma}\in\mathrm{Aut}(H)\)
such that \(\hat{\sigma}\circ\phi=\phi\circ\sigma\).
By induction we obtain \(\hat{\sigma}^n\circ\phi=\phi\circ\sigma^n\), for all \(n\in\mathbb{Z}\):
let \(n\ge 2\) be an integer and assume that \(\hat{\sigma}^{n-1}\circ\phi=\phi\circ\sigma^{n-1}\), then
\[
\hat{\sigma}^n\circ\phi
=\hat{\sigma}^{n-1}\circ\hat{\sigma}\circ\phi
=\hat{\sigma}^{n-1}\circ\phi\circ\sigma
=\phi\circ\sigma^{n-1}\circ\sigma=\phi\circ\sigma^n.
\]
Furthermore, \((\sigma^{-1})\,\hat{}=\hat{\sigma}^{-1}\).
Now let \(y\in H\).
Since \(\phi\) is surjective, there exists \(x\in G\) with \(\phi(x)=y\), and we obtain, as required,
\[
y^{T_{\hat{\sigma}}}
=
\prod_{i=0}^{d-1}\,\hat{\sigma}^i(y)
=
\prod_{i=0}^{d-1}\,\hat{\sigma}^i(\phi(x))
=
\prod_{i=0}^{d-1}\,\phi(\sigma^i(x))
=
\phi\left(\prod_{i=0}^{d-1}\,\sigma^i(x)\right)
=
\phi\left(x^{T_\sigma}\right)\in\phi(G^\prime)=H^\prime.
\]
\end{proof}

%
%
%
%

\begin{corollary}
\label{cor:ParentOperator}
In a descendant tree \(\mathcal{T}\) of finite \(p\)-groups with edges \(\pi:G\to\pi{G}\),
the property of \textbf{not} being a \(\sigma\)-group of degree \(d\)
is inherited from the parent \(\pi{G}\) by the immediate descendant \(G\).
\end{corollary}

\begin{proof}
The parent operator \(\pi:\,G\to\pi{G}\)
is the canonical projection from \(G\) onto the quotient \(\pi{G}=G/\gamma_c{G}\)
by the last non-trivial member \(\gamma_c{G}\), \(c=\mathrm{cl}(G)\),
of the lower central series \((\gamma_i{G})_{i\ge 1}\) of \(G\),
and thus \(\pi\) is an epimorphism with characteristic kernel \(\ker(\pi)=\gamma_c{G}\),
whence Theorem
\ref{thm:Epimorphism}
justifies the claim.
\end{proof}

%
%
%
%

\begin{remark}
\label{rmk:AnnTrace}
A \(\sigma\)-group \(G\) in the classical sense is a \(\sigma\)-group of degree \(2\) in the new sense,
since \(x\sigma(x)\in G^\prime\) is equivalent with \(\sigma(x)G^\prime=x^{-1}G^\prime\).
Such a group is also referred to as a group with \textit{generator inverting} automorphism
or briefly GI-automorphism.
\end{remark}

%
%
%
%

\begin{theorem}
\label{thm:CycQuart}
{\rm (i)} The \(p\)-class tower group \(G_p^{(\infty)}{{F}}\) and
all higher \(p\)-class groups \(G_p^{(n)}{{F}}\) with \(n\ge 2\) of a cyclic quartic number field \({F}\)
are \(\sigma\)-groups of degree \(4\).

{\rm (ii)} When the quadratic subfield \({k}<{F}\) has a trivial \(p\)-class group,
the groups \(G_p^{(\infty)}{{F}}\) and \(G_p^{(n)}{{F}}\) with \(n\ge 2\) are simultaneously \(\sigma\)-groups of degree \(2\).
\end{theorem}

\begin{proof} (i)
The generating automorphism \(\sigma\) of \(F/\mathbb{Q}\)
annihilates the class group \(\mathrm{Cl}(F)\)
when it acts by its trace \(T_\sigma=\sum_{i=0}^3\,\sigma^i\in\mathbb{Z}\lbrack\langle\sigma\rangle\rbrack\),
since 
$$x^{T_\sigma}=\prod_{i=0}^3\,\sigma^i(x)=\mathrm{N}_{F/\mathbb{Q}}(x)\in\mathrm{Cl}(\mathbb{Q})=1$$
for all \(x\in\mathrm{Cl}(F)\).
Of course, the same is true for all \(p\)-class groups \(\mathrm{Cl}_p(F)\) with primes \(p\).
Finally, for any \(n\in\mathbb{N}\cup\lbrace\infty\rbrace\),  we have  the isomorphisms
$$\mathrm{G}_p^n{F}/\left(\mathrm{G}_p^n{F}\right)^\prime\simeq\mathrm{Cl}_p(F).$$

(ii) When the unique (real) quadratic subfield \(k<F\) has trivial \(p\)-class group \(\mathrm{Cl}_p(k)=1\),
 the relative automorphism \(\tau=\sigma^2\in\mathrm{Gal}(F/k)\) with order \(2\)
acts by inversion on \(\mathrm{Cl}_p(F)\), since
$$x^{T_\tau}=x^{1+\tau}=x\cdot\tau(x)=\mathrm{N}_{F/k}(x)\in\mathrm{Cl}_p(k)=1,$$
and thus \(x^\tau=x^{-1}\)
for all \(x\in\mathrm{Cl}_p(F)\).
\end{proof}

%
%
%
%

\begin{remark}
\label{rmk:StrongSigma}
A pro-\(p\)-group \(G\) with finite abelianization \(G/G^\prime\)
is called a \textit{strong} \(\sigma\)-group
if it possesses an automorphism \(\sigma\) of order \(2\)
which acts as inversion on both cohomology groups \(\mathrm{H}^1(G,\mathbb{F}_p)\) and \(\mathrm{H}^2(G,\mathbb{F}_p)\).
We emphasize the following two facts:
\begin{itemize}
\item
An epimorphism does not necessarily preserve the property of being a strong \(\sigma\)-group.
\item
Whereas the group \(\mathrm{G}_p^{(\infty)}{F}\) of a quadratic field \(F\) is a strong \(\sigma\)-group,
according to Schoof
\cite[Lem. 4.1, p.\,217]{Sf},
this is not necessarily the case for a cyclic quartic field \(F\).
See for instance the unusual cases in Theorem
\ref{thm:UnusualSigmaImag}.
\end{itemize}
\end{remark}

In view of our special situation with \(p=5\), \(F=M\), \(\mathrm{Cl}_5(M)=(5,5)\) and \(k=k_0^+\),
we tested finite metabelian \(5\)-groups \(G\) with \(G/G^\prime\simeq (5,5)\) of order \(\lvert G\rvert=3125=5^5\) and coclass \(\mathrm{cc}(G)=2\),
for the property of simultaneously being a \(\sigma\)-group of degree \(4\) and degree \(2\).
These groups are crucial contestants for second \(5\)-class groups \(\mathrm{G}_5^{(2)}{M}\)
and form the stem of Hall's isoclinism family \(\Phi_6\).
(See
\cite[\S3.5, pp.\,445--448]{Ma4}
and
\cite[\S 7, pp. \,93--98]{Ma15}.)
In Table
\ref{tbl:TrfKerStem5Icl6},
the groups are characterized by their identifiers according to James
\cite{Jm}
and the SmallGroups Library
\cite{BEO2}.
An asterisk \(\ast\) marks a Schur \(\sigma\)-group,
and a flag \(f\in\lbrace 0,1\rbrace\) indicates a \(\sigma\)-group of simultaneous degrees \(4\) and \(2\).

%
%
%
%

\renewcommand{\arraystretch}{1.1}
\vskip 3mm 

\begin{table}[ht]
\caption{The Artin pattern of the twelve \(5\)-groups of order \(5^5\) in the stem of \(\Phi_6\)}
\label{tbl:TrfKerStem5Icl6}
\begin{center}
\begin{tabular}{|ll|c|ccc|}
\hline
 \multicolumn{2}{|c|}{Identifier of the \(5\)-Group} & Flag  & \multicolumn{3}{|c|}{\(5\)-Principalization Type}          \\
 James                  & SmallGroup                 & \(f\) & \(\varkappa\) & Cycle Pattern  & Property                  \\
\hline
 \(\Phi_6(2^21)_a\)     & \(\langle 3125,14\rangle\ast\) & 1 & \((123456)\) & \((1)(2)(3)(4)(5)(6)\) & identity           \\
 \(\Phi_6(2^21)_{b_1}\) & \(\langle 3125,11\rangle\ast\) & 1 & \((125364)\) & \((1)(2)(3564)\)       & \(4\)-cycle        \\
 \(\Phi_6(2^21)_{b_2}\) & \(\langle 3125,7\rangle\)      & 1 & \((126543)\) & \((1)(2)(36)(45)\)     & two \(2\)-cycles   \\
 \(\Phi_6(2^21)_{c_1}\) & \(\langle 3125,8\rangle\ast\)  & 0 & \((612435)\) & \((16532)(4)\)         & \(5\)-cycle        \\
 \(\Phi_6(2^21)_{c_2}\) & \(\langle 3125,13\rangle\ast\) & 0 & \((612435)\) & \((16532)(4)\)         & \(5\)-cycle        \\
 \(\Phi_6(2^21)_{d_0}\) & \(\langle 3125,10\rangle\)     & 0 & \((214365)\) & \((12)(34)(56)\)       & three \(2\)-cycles \\
 \(\Phi_6(2^21)_{d_1}\) & \(\langle 3125,12\rangle\ast\) & 0 & \((512643)\) & \((154632)\)           & \(6\)-cycle        \\
 \(\Phi_6(2^21)_{d_2}\) & \(\langle 3125,9\rangle\ast\)  & 0 & \((312564)\) & \((132)(456)\)         & two \(3\)-cycles   \\
\hline
 \(\Phi_6(21^3)_a\)     & \(\langle 3125,4\rangle\)      & 1 & \((022222)\) &                        &nrl. const. with fp.\\
 \(\Phi_6(21^3)_{b_1}\) & \(\langle 3125,5\rangle\)      & 1 & \((011111)\) &                        & nearly constant    \\
 \(\Phi_6(21^3)_{b_2}\) & \(\langle 3125,6\rangle\)      & 1 & \((011111)\) &                        & nearly constant    \\
 \(\Phi_6(1^5)\)        & \(\langle 3125,3\rangle\)      & 1 & \((000000)\) &                        & constant           \\
\hline
\end{tabular}
\end{center}
\end{table}

%
%
%
%

%
%
%
%

\begin{theorem}
\label{thm:Sigma5GroupsDegree4}
A finite \(5\)-group \(G\) with \(G/G^\prime\simeq (5,5)\) which is a \(\sigma\)-group of degree \(4\)
is either of coclass \(\mathrm{cc}(G)=1\)
or isomorphic to one of the two Schur \(\sigma\)-groups \(\langle 3125,i\rangle\) with \(i\in\lbrace 11,14\rbrace\)
or isomorphic to a descendant of one of the capable groups
\(\langle 3125,i\rangle\) with \(i\in\lbrace 3,4,5,6,7\rbrace\).
\end{theorem}

\begin{proof}
Using permutation representations,
we compiled a program script in Magma
\cite{MAGMA}
for testing whether an assigned \(5\)-group \(G\) with \(G/G^\prime\simeq (5,5)\) is a \(\sigma\)-group of degree \(4\).
\end{proof}

%
%
%
%

%
%
%
%
%
%
%
%
\vskip 5mm

\section{Frobenius and non-Galois unramified \(5\)-extensions}
\label{s:FrobeniusExtensions}


\subsection{On the cyclic quartic fields \(M\)}
\label{ss:CyclicQuartic}
\noindent
Let \(\zeta_{5}\) be a primitive \(5\)th root of unity,
then the irreducible polynomial of \(\zeta_{5}\) is given by
\(\mathrm{Irr}_{\mathbb{Q}}(\zeta_5)=X^4+X^3+X^2+X+1\),
and \(\mathrm{Gal}(\mathbb{Q}(\zeta_5)/\mathbb{Q})=\langle\mu\rangle\)
is a cyclic group of order \(4\)
which admits one subgroup \(\langle\mu^2\rangle\) of order \(2\).
By Galois correspondence, this subgroup corresponds to
\(\mathbb{Q}(\zeta_5)^+=\mathbb{Q}(\zeta_5+\zeta_5^{-1})=\mathbb{Q}(\sqrt{5})\).
(See Figure~\ref{fig:Scheme1}.)

%
%
%
%
%
%
%
%

\begin{figure}[ht]
\caption{Galois correspondence between \(\mathrm{Gal}(k_0/\mathbb{Q})\) and \(k_0\)}
\label{fig:Scheme1}
\setlength{\unitlength}{0.8cm}
\begin{picture}(8,5.5)(-2,1)

    \put(-4,5.5){\makebox(0,0)[cb]{\(1\)}}
    \put(-5,4.5){\makebox(0,0)[rc]{\(2\)}}
    \put(-5,2.5){\makebox(0,0)[rc]{\(2\)}}
    \put(-4,3.7){\line(0,1){1.5}}
    \put(-4.3,3.2){\makebox(0,0)[cb]{\(H=\langle\mu^{2}\rangle\)}}
    \put(-4,1.5){\line(0,1){1.5}}
    \put(-4,1){\makebox(0,0)[cb]{\(\langle\mu\rangle\)}}

    \put(4,5.5){\makebox(0,0)[cb]{\(k_0=\mathbb{Q}(\zeta_{5})\)}}
    \put(4,3.7){\line(0,1){1.5}}
    \put(7,3.2){\makebox(0,0)[cb]{\(k_0^+=\mathrm{inv}(H)=\mathbb{Q}(\zeta_{5}+\zeta_{5}^{-1})=\mathbb{Q}(\sqrt{5})\)}}
    \put(4,1.5){\line(0,1){1.5}}
    \put(4,1){\makebox(0,0)[cb]{\(\mathbb{Q}\)}}

    \put(0,3.2){\makebox(0,0)[cb]{Galois}}
    \put(0,2.8){\makebox(0,0)[cb]{\(\vector(1,0){2}\)}}
    \put(0,2.8){\makebox(0,0)[cb]{\(\vector(-1,0){2}\)}}

\end{picture}
\end{figure}
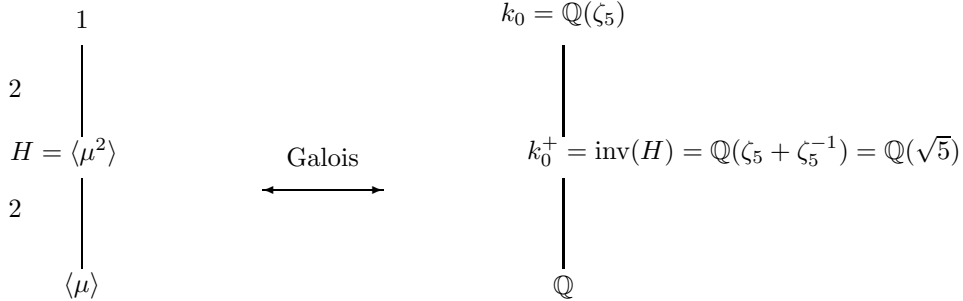

%
%
%
%
%
%
%
%

Let \(d\) be a square-free  integer  prime to \(5\). Then
\({L}=\mathbb{Q}(\sqrt{d},\zeta_{5})\) is a normal
extension over \(\mathbb{Q}\) of degree \(8\),  and the Galois group is
\[
\mathrm{Gal}\left({L}/\mathbb{Q}\right)
=
\langle\tau,\mu\rangle
=
\lbrace 1,\tau,\mu,\mu^{2},\mu^{3},\tau\mu,\tau\mu^{2},\tau\mu^{3}\rbrace,\text{ where } \tau(\sqrt{d})=-\sqrt{d}.
\]
This is an abelian group of type \((2,4)\)
which  has six proper subgroups ordered as follows :
\[
H_{1}=\langle\tau\rangle,
\ H_{2}=\langle\mu\rangle,
\ H_{3}=\langle\mu^{2}\rangle,
\ H_{4}=\langle\tau\mu\rangle,
\ H_{5}=\langle\tau\mu^{2}\rangle
\text{ and }
H_{6}=\langle\tau,\mu^{2}\rangle.
\]
Note that the subgroups \(H_{1},\ H_{3},\ H_{5}\) are cyclic of order \(2\), the subgroups \(H_{2},\ H_{4}\) are cyclic of order
\(4\), and the group \(H_{6}\) is bicyclic of order \(4\). (See Figure~\ref{fig:Scheme2}.)

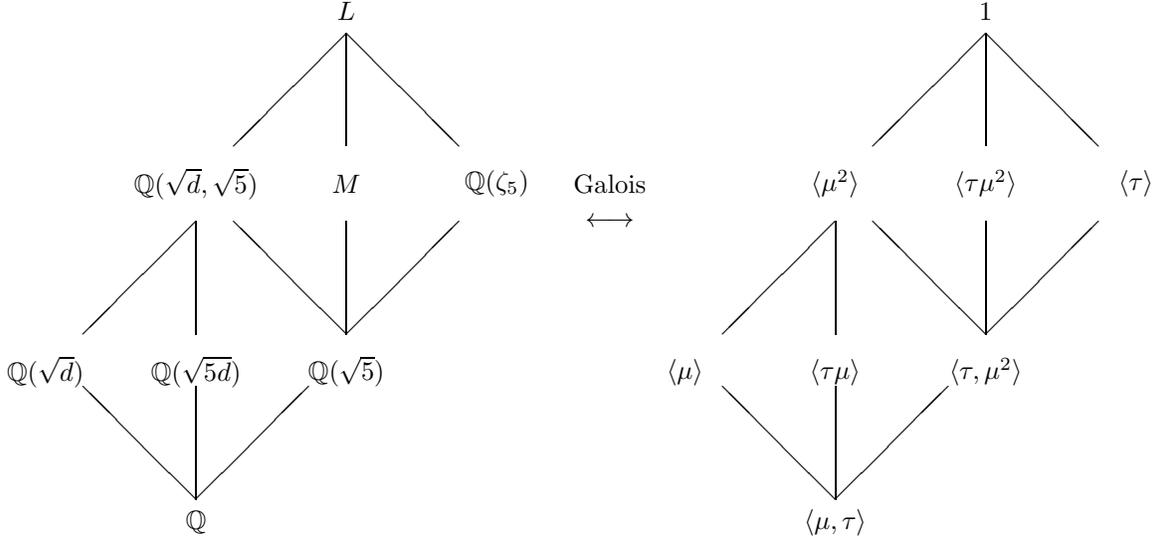
\begin{figure}[hb]
\caption{Galois correspondence between \(L\) and \(\mathrm{Gal}(L/\mathbb{Q})\)}
\label{fig:Scheme2}

\setlength{\unitlength}{1cm}
\begin{picture}(12,7.5)(-3,-0.5)

 \put(0,6.3){\makebox(0,0)[cc]{\(L\)}}

 \put(0,6){\line(-1,-1){1.5}}
 \put(0,6){\line(1,-1){1.5}}
 \put(0,6){\line(0,-1){1.5}}

 \put(2,4){\makebox(0,0)[cc]{\(\mathbb{Q}(\zeta_{5})\)}}
 \put(0,4){\makebox(0,0)[cc]{\(M\)}}
 \put(-2,4){\makebox(0,0)[cc]{\(\mathbb{Q}(\sqrt{d},\sqrt{5})\)}}

 \put(0,2){\line(1,1){1.5}}
 \put(0,2){\line(-1,1){1.5}}
 \put(0,2){\line(0,1){1.5}}

 \put(-2,3.5){\line(-1,-1){1.5}}
 \put(-2,3.5){\line(0,-1){1.5}}

 \put(0,1.5){\makebox(0,0)[cc]{\(\mathbb{Q}(\sqrt{5})\)}}
 \put(-2,1.5){\makebox(0,0)[cc]{\(\mathbb{Q}(\sqrt{5d})\)}}
 \put(-4,1.5){\makebox(0,0)[cc]{\(\mathbb{Q}(\sqrt{d})\)}}

 \put(-2,-0.2){\line(1,1){1.5}}
 \put(-2,-0.2){\line(-1,1){1.5}}
 \put(-2,-0.2){\line(0,1){1.5}}

 \put(-2,-0.5){\makebox(0,0)[cc]{\(\mathbb{Q}\)}}

 \put(3.5,4){\makebox(0,0)[cc]{Galois}}
 \put(3.5,3.5){\makebox(0,0)[cc]{\(\longleftrightarrow\)}}

 \put(8.5,6.3){\makebox(0,0)[cc]{\(1\)}}

 \put(8.5,6){\line(-1,-1){1.5}}
 \put(8.5,6){\line(1,-1){1.5}}
 \put(8.5,6){\line(0,-1){1.5}}

 \put(10.5,4){\makebox(0,0)[cc]{\(\langle\tau\rangle\)}}
 \put(8.5,4){\makebox(0,0)[cc]{\(\langle\tau\mu^{2}\rangle\)}}
 \put(6.5,4){\makebox(0,0)[cc]{\(\langle\mu^{2}\rangle\)}}

 \put(8.5,2){\line(1,1){1.5}}
 \put(8.5,2){\line(-1,1){1.5}}
 \put(8.5,2){\line(0,1){1.5}}

 \put(6.5,3.5){\line(-1,-1){1.5}}
 \put(6.5,3.5){\line(0,-1){1.5}}

 \put(8.5,1.5){\makebox(0,0)[cc]{\(\langle\tau,\mu^{2}\rangle\)}}
 \put(6.5,1.5){\makebox(0,0)[cc]{\(\langle\tau\mu\rangle\)}}
 \put(4.5,1.5){\makebox(0,0)[cc]{\(\langle\mu\rangle\)}}

 \put(6.5,-0.2){\line(1,1){1.5}}
 \put(6.5,-0.2){\line(-1,1){1.5}}
 \put(6.5,-0.2){\line(0,1){1.5}}

 \put(6.5,-0.5){\makebox(0,0)[cc]{\(\langle\mu,\tau\rangle\)}}

\end{picture}
\end{figure}

%
%
%
%
%
%
%
%

We consider the field \(M\) fixed by the subgroup \(\langle\tau\mu^2\rangle\).
Then \(M\) is a cyclic quartic field
and can be generated by adjunction \(M=\mathbb{Q}(\alpha)\)   of
$$\alpha=(\zeta_5-\zeta_5^{-1})\sqrt{d}=\sqrt{-\frac{5d}{2}-\frac{d}{2}\sqrt{5}}$$
 to \(\mathbb{Q}.\)
With respect to the quartic cyclotomic mirror field \(k_0=\mathbb{Q}(\zeta_5)\),
\(M\) satisfies the \textit{quintic reflection theorem} (Equation
\eqref{eqn:5Reflection}).

%
%
%
%

\begin{lemma}
\label{lem:cycfiel} {\rm (i)}
Let \({K}\) be a number field and \({F}/{K}\) be a  cyclic quartic
extension.   Then there exist \(n,\, e\) and \(f\neq 0\) in \({K}\) such that
\begin{enumerate}
\item
\(n\) is not a square in \({K}\),
\item
\(n(e^{2}-f^{2}n)\) is a square in \({K}\),
\item
\({F}={K}\left(\alpha \right)\), where \(\alpha=\sqrt{e+f\sqrt{n}}\),
\end{enumerate}
and the minimal polynomial of \(\alpha\) over \({K}\) is given by
\(
\mathrm{Irr}_{K}\left(\alpha\right)=X^{4}-2eX^{2}+(e^{2}-f^{2}n).
\)

{\rm (ii)} Conversely, if there exist numbers \(n,\ e \text{ and } f\neq0\) in \(K\)
which satisfy the conditions \((1),\,(2)\) and \((3)\), then
\({F}/{K}\) is a  cyclic quartic extension and the
polynomial \(P(X)=X^{4}-2eX^{2}+(e^{2}-f^{2}n)\) is irreducible over
\({K}\). In fact, \({F}\) is the splitting field of \(P(X)\).
\end{lemma}

\begin{proof} (i)
It is known that the group \(\mathbb{Z}/4\mathbb{Z}\)
has a single subgroup of order \(2\).
By Galois theory,
there exists a corresponding intermediary field \(R\)
of the cyclic quartic extension \(F/K\).
Thus we can find \(n\in K\), which is not a  square in \(K\),
such that \(R=K(\sqrt{n})\).
Since \(F\) is a quadratic extension of \(R\),
there exists \(\alpha\in F\) such that
\(\alpha^{2}=e+f\sqrt{n}\in R\) with \(e,f\in K\), \(f\ne 0\),
and \(F=R(\alpha)\).
Thus we have
\(K(\alpha)=F\), because \(\alpha\not\in R\).
Furthermore, it is obvious that
the minimal polynomial of \(\alpha\) is \(P(X)=X^{4}-2eX^{2}+(e^{2}-f^{2}n)\)
and the splitting field of \(P(X)\) over \(\mathbb{Q}\) is \(F\).

The discriminant of \(P\) is given by
$$D=16(e^{2}-f^{2}n)(2f\sqrt{n})^{4}=2^8f^{4}n^{2}(e^{2}-f^{2}n).$$
Therefore the Galois group \(\mathrm{Gal}\left(F/K\right)\)
can be seen as a subgroup of the permutation group of the roots of \(P(X)\),
which is isomorphic to \(S_{4}\),
and cannot be injected into \(\mathcal{A}_{4}\),
since the group \(\mathcal{A}_{4}\) does not have a cyclic subgroup of order \(4\).
We conclude that the discriminant is not a square in \(K\),
whence \(K\left(\sqrt{e^{2}-f^{2}n}\right)/K\) is of degree \(2\)
and is contained in \(F\).
It follows that
$$R=K\left(\sqrt{n}\right)=K\left(\sqrt{e^{2}-f^{2}n}\right),$$
so \(\frac{e^{2}-f^{2}n}{n}\) is a square in \(K\).
Consequently, we see that
\(n\left(e^{2}-f^{2}n\right)=n^{2}\frac{e^{2}-f^{2}n}{n}\)
is a square in \(K\).

(ii) Conversely, let 
$$P(X)=X^{4}-2eX^{2}+(e^{2}-f^{2}n)$$
with \(n,e,f\in K\), \(f\ne 0\),
such that the conditions \((1),\,(2)\) and \((3)\) are satisfied.
Since \(\alpha\) is a root of \(P(X)\),
the degree \(\lbrack F:K\rbrack\) must be a divisor of \(4\).
Since \(K\left(\sqrt{n}\right)\subseteq F\), we have
either \(F=K\left(\sqrt{n}\right)\) or \(\lbrack F:K\rbrack=4\).
If we have \(F=K\left(\sqrt{n}\right)\),
there exist \(u,v \in K\) such that
\(\sqrt{e+f\sqrt{n}}=u+v\sqrt{n}\).
Thus 
$$e^{2}-f^{2}n=\left(u^{2}-v^{2}n\right)^{2},$$
and by \((2)\) we conclude that \(n\) is a square in \(K\), which is a contradiction.
So \(\lbrack F:K\rbrack=4\), and this enforces that
\(P(X)\) is the minimal polynomial of \(\alpha\).
From the fact
$$e-f\sqrt{n}=\left(\frac{1}{\sqrt{n}}\right)^{2}\frac{n(e^{2}-f^{2}n)}{\alpha^{2}},$$
we conclude that \(F\) is the splitting field of \(P(X)\) over \(K\).
Moreover, \(F\) is normal and \(\#\mathrm{Gal}\left(F/K\right)=4\).
Now we prove that \(\mathrm{Gal}\left(F/K\right)\) is cyclic of order \(4\).
If the Galois group \(\mathrm{Gal}\left(F/K\right)\) were isomorphic to \(V_{4}\),
then the discriminant would be a square in \(K\).
This would imply that \(n\) were a square in \(K\), which is a contradiction.
\end{proof}

%
%
%
%
%
%
%
%
%
%
%
%

%
%
%
%
%
%
%
%

\begin{corollary}
\label{cor:MinPol}
 Let \(d\) be  a square-free integer prime to \(5\) and \(\zeta_{5}\) be a primitive
\(5th\) root of unity. Then the mirror image \({M }\) of
\({k}_{1}=\mathbb{Q}(\sqrt{d})\) can always be generated by adjoining the algebraic number
$$\alpha = (\zeta_{5} - \zeta_{5}^{-1})\sqrt{d}= \sqrt{\frac{-5d}{2}+\frac{-d}{2}\sqrt{5}} $$
 to the rational,
whence \({M}\) is complex for \(d > 0\) and \({M}\) is real for \(d <0\).
For the construction of \({M}\) one can therefore use the minimal
polynomial of \(\alpha\) over \(\mathbb{Q}\), which is given by
\begin{equation}
\label{eqn:MinPol}
\mathrm{Irr}_\mathbb{Q}(\alpha)=X^4+5dX^2+5d^2.
\end{equation}
\end{corollary}

%
%
%
%

\begin{remark}
\label{rmk:CondAndDisc}
The conductor \(c(M)\) and the discriminant \(d(M)\) of \(M\) are given by
\begin{eqnarray}
\label{eqn:CondAndDisc}
c(M)&=&
\begin{cases}
20d & \text{ if } d\equiv 2,3\pmod{4}, \\
 5d & \text{ if } d\equiv 1\pmod{4},
\end{cases}\\
d(M)&=&c(M)^2d(k_0^+)=
\begin{cases}
2000d^2 & \text{ if } d\equiv 2,3\pmod{4}, \\
 125d^2 & \text{ if } d\equiv 1\pmod{4},
\end{cases}
\end{eqnarray}
where \(d(k_0^+)=5\) is the discriminant of the quadratic subfield \(k_0^+=\mathbb{Q}(\sqrt{5})\) of \(M\).
(See
\cite{HHRWH,SpWi}.)
\end{remark}

%
%
%
%
%
%
%
%

\subsection{Imaginary cyclic quartic fields \(M\) with \(d>0\)}
\label{ss:ImaginaryTheory}

\noindent
In the following,
the two Frobenius groups \(F_{5,w}\) of order \(20\) with primitive root \(w\in\{2,3\}\) modulo \(5\) will be denoted by
\begin{equation}
\label{eqn:5Frobenius}\left\{
\begin{aligned}
F_{5,2} &= \langle\ \sigma,\iota\mid\sigma^5=1,\ \iota^4=1, \iota^{-1}\sigma\iota=\sigma^2\ \rangle, \\
F_{5,3} &= \langle\ \sigma,\iota\mid\sigma^5=1,\ \iota^4=1, \iota^{-1}\sigma\iota=\sigma^3\ \rangle,
\end{aligned} \right.
\end{equation}
where \(\iota\vert_{M}=\mu\vert_{M}\).

%
%
%
%

\begin{proposition}
\label{prp:FrobeniusFieldsImag}
Let \({E}_1,\ldots,{E}_6\) be the six unramified cyclic quintic extensions
of the imaginary cyclic quartic field \({M}=\mathbb{Q}\left((\zeta_5-\zeta_5^{-1})\sqrt{d}\right)\), \(d>0\),
with \(5\)-class group \(\mathrm{Cl}_5(M)\simeq C_5\times C_5\)
or, more generally, of \(5\)-class rank \(2\).
The properties of these fields as absolute extensions \({E}_i/\mathbb{Q}\),
in dependence on the eight cases in Table
\ref{tbl:5RankCasesImag},
are given as follows:

$(1)$
In cases \((\mathrm{a})\) and \((\mathrm{g})\), all six fields \({E}_1,\ldots,{E}_6\) are normal and share isomorphic automorphism groups \(\mathrm{Gal}({E}_i/\mathbb{Q})\simeq {F}_{5,2}\) for \( i=1, \dots , 6\).

$(2)$
In cases \((\mathrm{b})\) and \((\mathrm{h})\), all six fields \({E}_1,\ldots,{E}_6\) are normal and share isomorphic automorphism groups \(\mathrm{Gal}({E}_i/\mathbb{Q})\simeq {F}_{5,3}\) for \(i=1, \dots ,  6\).

$(3)$
In all the other cases \((\mathrm{c})\), \((\mathrm{d})\), \((\mathrm{e})\), \((\mathrm{f})\),
two extensions are normal with non-isomorphic automorphism groups, say
$$\mathrm{Gal}({E}_1/\mathbb{Q})\simeq {F}_{5,2} \;\mbox{ and }\;
\mathrm{Gal}({E}_2/\mathbb{Q})\simeq {F}_{5,3},$$
but the other four extensions are non-Galois and form two conjugate pairs
\({E}_3\simeq {E}_4\) and \({E}_5\simeq {E}_6\).
\end{proposition}


\renewcommand{\arraystretch}{1.0}

\begin{table}[ht]
\caption{All possible \(5\)-class ranks \(r_1:=\varrho_5(k_1)\), \(r_2:=\varrho_5(k_2)\)
and invariants \(\delta_1\), \(\delta_2\) for the associated quadratic fields \(k_1,k_2\)
which are \(5\)-dual to an imaginary cyclic quartic field 
\(M=\mathbb{Q}((\zeta_5-\zeta_5^{-1})\sqrt{d})\), \(d>0\),
with \(5\)-class rank \(r:=\varrho_5(M)=2\)}
\label{tbl:5RankCasesImag}
\begin{center}
\begin{tabular}{|c||c|c|c|c|c|c|c|}
\hline
 Case & \(r_1\) & \(\delta_1\) & \(r_2\) & \(\delta_2\) \\
\hline
  (a) &  \(1\) &         \(0\) &   \(0\) &        \(1\) \\
  (b) &  \(0\) &         \(1\) &   \(1\) &        \(0\) \\
  (c) &  \(1\) &         \(1\) &   \(1\) &        \(1\) \\
  (d) &  \(0\) &         \(0\) &   \(0\) &        \(0\) \\
  (e) &  \(1\) &         \(1\) &   \(0\) &        \(0\) \\
  (f) &  \(0\) &         \(0\) &   \(1\) &        \(1\) \\
  (g) &  \(2\) &         \(1\) &   \(0\) &        \(1\) \\
  (h) &  \(0\) &         \(1\) &   \(2\) &        \(1\) \\
\hline
\end{tabular}
\end{center}
\end{table}

%
%
%
%

\begin{proof}
According to the quintic reflection theorem
\cite{Ki},
the assumption \(r=2\) implies that one of the eight disjoint cases in Table
\ref{tbl:5RankCasesImag}
is satisfied.
 
In case \((\mathrm{a})\), the \(5\)-Selmer group of \({k}_1\) is given by
\(V_5({k}_1)=\langle\alpha_{11},\varepsilon_1\rangle\). See
\cite[p.\,2, l.\,3]{Ki}.
Let \
$${E}_1:=\mathrm{Spl}_{\mathbb{Q}}f(X,\alpha_{11})\; \mbox{  and  }\; {E}_2:=\mathrm{Spl}_{\mathbb{Q}}f(X,\varepsilon_1).$$
In virtue of \(\delta_1=0\), \({E}_1\) and \({E}_2\) are unramified cyclic quintic extensions of \({M}\).
According to
\cite[p.\,17, l.\,9--23]{Ki},
\(\mathrm{Gal}({E}_i/\mathbb{Q})\simeq {F}_{5,2}\), for \(1\le i\le 2\).
Let \({L}:={E}_1\cdot {E}_2\) be the compositum.
Then, by
\cite[Lem. 2.5]{Ki},
all proper subextensions \({E}\) of \({L}/{M}\) have \(\mathrm{Gal}({E}/\mathbb{Q})\simeq {F}_{5,2}\).

In case \((\mathrm{b})\), the \(5\)-Selmer group of \({k}_2\) is given by
\(V_5({k}_2)=\langle\alpha_{21},\varepsilon_2\rangle\). See
\cite[p.\,2, l.\,3]{Ki}.
Let \({E}_1:=\mathrm{Spl}_{\mathbb{Q}}f(X,\alpha_{21})\) and \({E}_2:=\mathrm{Spl}_{\mathbb{Q}}f(X,\varepsilon_2)\).
In virtue of \(\delta_2=0\), \({E}_1\) and \({E}_2\) are unramified cyclic quintic extensions of \({M}\).
According to
\cite[p.\,17, l.\,9--23]{Ki},
\(\mathrm{Gal}({E}_i/\mathbb{Q})\simeq {F}_{5,3}\), for \(1\le i\le 2\).
Let \({L}:={E}_1\cdot {E}_2\) be the compositum.
Then, by
\cite[Lem. 2.5]{Ki},
all proper subextensions \({E}\) of \({L}/{M}\) have \(\mathrm{Gal}({E}/\mathbb{Q})\simeq {F}_{5,3}\).

Exemplarily, we consider case  \((\mathrm{d})\).
Then the \(5\)-Selmer groups of \({k}_1\) and \({k}_2\) are given by
\mbox{\(V_5({k}_1)=\langle\varepsilon_1\rangle\)}, \(V_5({k}_2)=\langle\varepsilon_2\rangle\).
Let 
$${E}_1:=\mathrm{Spl}_{\mathbb{Q}}f(X,\varepsilon_1) \;\mbox{ and }\; {E}_2:=\mathrm{Spl}_{\mathbb{Q}}f(X,\varepsilon_2).$$
Then, in virtue of \(\delta_1=\delta_2=0\),
\({E}_1\) and \({E}_2\) are unramified cyclic quintic extensions of \({M}\).
According to
\cite[p.\,17, l.\,9--23]{Ki},
$$\mathrm{Gal}({E}_1/ \mathbb{Q})\simeq {F}_{5,2}\; \mbox{ and }\;
\mathrm{Gal}({E}_2/ \mathbb{Q})\simeq {F}_{5,3}.$$
Let \({L}:={E}_1\cdot {E}_2\) be the compositum.
Then \({E}/{M}\) is also an unramified cyclic quintic extension, for any proper subextension \({E}\) of \({L}/{M}\) distinct from \({E}_1\) and \({E}_2\).
Assume that \(\mathrm{Gal}({E}/\mathbb{Q})\simeq F_{{5},2}\).
Since \({L}={E}_1\cdot {E}\), all proper subextensions \({E}^\prime\) of \({L}/{M}\) have \(\mathrm{Gal}({E}^\prime/\mathbb{Q})\simeq {F}_{5,2}\), by
\cite[Lem. 2.5]{Ki}.
This is a contradiction to \(\mathrm{Gal}({E}_2/\mathbb{Q})\simeq {F}_{5,3}\).
In the same manner, the assumption that \(\mathrm{Gal}({E}/\mathbb{Q})\simeq {F}_{5,3}\) leads to a contradiction.
Therefore \({E}/\mathbb{Q}\) must be a non-Galois extension.
\end{proof}

%
%
%
%
%
%
%
%
%
%
%
%

\subsection{Infinite family of imaginary cyclic quartic fields \(M\) whose \(5\)-rank is at least \(2\)}  
\label{ss:InfiniteFamily}



As before, let \(d\neq1\) be a square-free integer prime to \(5\), and let
\({k}_{1}=\mathbb{Q}(\sqrt{d})\) and
\({k}_{2}=\mathbb{Q}(\sqrt{5d})\) be the associated quadratic
fields. For \(\gamma \in k={k}_{i},\,\, i\in\{1,2\}\),
 Y. Kishi \cite[p.\,6]{Ki} has defined the polynomial
\[
f(X,\gamma)
=
X^{5}-5\mathrm{N}_{{k}}(\gamma)X^{3}+5\mathrm{N}_{{k}}(\gamma)^{2}X-\mathrm{N}_{{k}}(\gamma)^2\mathrm{Tr}_{{k}}(\gamma),
\]
 where \(\mathrm{N}_{{k}}, \mathrm{Tr}_{{k}}\) are the norm map and the trace map of
 \({k}/\mathbb{Q}\).
 The minimal splitting field of \(f(X,\gamma)\) is noted  by \({K}_{\gamma}\).
  Furthermore, M. Imaoka and Y. Kishi
   \cite{ImKi},
have characterized all \(F_{5,w}\)-extensions with \(w\in\lbrace 2,3\rbrace\) as
 \({K}_{\gamma}\) for a suitable elements \(\gamma \in {k}_{i}\) with \(i\in\{1,2\}\).
 If \(\gamma\in k_1\), then
\(\mathrm{Gal}\left({K}_{\gamma}/\mathbb{Q}\right)\simeq{F}_{5,2} \), and if \(\gamma\in k_2\), then \(\mathrm{Gal}\left({K}_{\gamma}/\mathbb{Q}\right)\simeq{F}_{5,3}\).
Now, we consider the real quadratic fields
\(k_1=\mathbb{Q}(\sqrt{d})\)
and
\(k_2=\mathbb{Q}(\sqrt{5d})\),
\(d=(\alpha+\beta)^{2}-4\),
given by Kishi in
\cite[Ex. 3.5, p.\,489]{Ki2}
for \(p=5\), where the pair of integers
\((\alpha,\beta)\in\mathbb{N}\times\mathbb{N}\) such that
\(\alpha\geq 2,\ \beta\geq 2\) satisfies the simultaneous conditions
\begin{equation}
\label{eqn:Parameters}
\begin{cases}
\alpha^{2}-5^3\beta^{2}=4, \\
 \alpha+\beta\equiv 0\pmod{5^{2}}.
\end{cases}
\end{equation}

%
%
%
%

\begin{remark}
\label{rmk:InfiniteFamily}
The Pellian equation \(\alpha^{2}-5^3\beta^{2}=4\)
has infinitely many solutions \((\alpha,\beta)\),
which correspond to the powers \(\eta^n=\frac{\alpha+\beta\sqrt{5^3}}{2}\) of the
normpositive fundamental unit
 $$\eta=\frac{123+11\sqrt{5^3}}{2}=\frac{123+11\cdot 5\sqrt{5}}{2}$$
of the suborder with conductor \(f=5\) of \(\mathbb{Q}(\sqrt{5})\).
The solution \((\alpha,\beta)\) satisfies the additional constraint \(\alpha+\beta\equiv 0\bmod 5^{2}\) in
\eqref{eqn:Parameters}
if and only if \(n=7+5^2k\) with an integer \(k\ge 0\).
\end{remark}

%
%
%
%

\begin{proposition}
\label{prp:InfiniteFamily}
Let 
$$M=\mathbb{Q}\left((\zeta_{5}-\zeta_{5}^{-1})\sqrt{(\alpha+\beta)^{2}-4}\right),$$
where \(\alpha,\,\beta\) satisfy the conditions
\eqref{eqn:Parameters}.
Then the \(5\)-rank of the class group of \(M\) is greater than or equal to \(2\).
\end{proposition}

\begin{proof}
Let 
$$\epsilon_{1}=\frac{\alpha+\beta+\sqrt{d}}{2},\;
\mbox{ resp.  }\; \epsilon_{2}=\frac{\alpha+5^3\beta+5\sqrt{5d}}{2},$$
 be an element of 
$${k}_{1}= \mathbb{Q}\left( \sqrt{(\alpha+\beta)^{2}-4}\right),  \; \mbox{ resp. } \; {k}_{2}=\mathbb{Q}\left(\sqrt{5((\alpha+\beta)^{2}-4)}\right).$$ 
According to
\cite[Ex.\,3.5, p.\,489]{Ki2}, \(\epsilon_{1}\) and
\(\epsilon_{2}\) are units of \({k}_{1}\) and \({k}_{2}\), respectively.
They satisfy the conditions
\begin{equation}
\label{eqn:KishiUnits}
\begin{cases}
\mathrm{N}_{\mathbb{Q}(\sqrt{d})}\left(\epsilon_{1}^{2}\right)
=
\mathrm{N}_{\mathbb{Q}(\sqrt{5d})}\left(\epsilon_{2}\right)=1, \\[2mm]
\mathrm{Tr}_{\mathbb{Q}(\sqrt{d})}\left(\epsilon_{1}^{2}\right)
\equiv
\mathrm{Tr}_{\mathbb{Q}(\sqrt{5d})}\left(\epsilon_{2}\right)
\equiv\pm2\pmod{5^{3}}.
\end{cases}
\end{equation}
By applying
  \cite[Th.\,1.1, p.\,482, Prop.\,3.1, p.\,487]{Ki2},
   we prove that \({K}_{\epsilon_{1}^{2}}\) and
  \({K}_{\epsilon_{2}}\) are two different absolute Galois \({F}_{5}\)-extensions, unramified over \({M}\); it suffices to show that
  \(\epsilon_{1}^{2}\), resp. \(\epsilon_{2}\), cannot be the fifth power of an  element of
\({k}_{1}\), resp. \({k}_{2}\).

According to
\cite[Lem.~1, p.\,16 ]{Na},
we have the following general fact:
Let \(p\) be a prime number and \(\xi\) be an
  element of \(\mathbb{Q}(\sqrt{\delta})\) such
  that  \(\xi=\frac{u+v\sqrt{\delta}}{2}\). If
\(
0<|v|<\frac{\delta^{(p-1)/2}}{2^{p-1}}
\),
then
\(
\xi \not\in\mathbb{Q}(\sqrt{\delta})^{p}.
\)

Let us apply this result to \(\epsilon_{2}\) and \(\epsilon_{1}^{2}\).
By the assumptions
\eqref{eqn:Parameters},
\(\alpha +\beta=5^{2}c\), for some \(c\geq 1\). Hence
\((\alpha +\beta)^{2}=5^{4}c^{2}\)
and
\(5(\alpha +\beta)^{2}=5^{5}c^{2} \).
  Furthermore, \(5^{5}c^{2}\geq5^5> 36 \) and thus
\(5(\alpha +\beta)^{2}-20>16\),
whence
$$5 d> 16,\;
(5d)^2>16^2\;\mbox{ and }\;
\frac{(5 d)^{2}}{16}> 16.$$
Finally \(5<16<\frac{(5d)^{2}}{2^{4}}\),
  and if we put
  \(v:=5\)
  and
  \(\delta:=5d\),
  then
  \(v<\frac{\delta^2}{2^4}\),
  whence \(\epsilon_{2}\) cannot be the fifth power
  of an element in \({k}_{2}\).

  For \(\epsilon_{1}^{2}\), we express the square in the form
  $$\epsilon_{1}^{2}=\frac{\frac{(\alpha+\beta)^{2}+d}{2}+(\alpha+\beta)\sqrt{d}}{2}.$$
  Moreover we have,
\[
\alpha+\beta < \frac{d^{2}}{16}
\Longleftrightarrow
  16(\alpha+\beta)<
  (\alpha+\beta)^{4}-8(\alpha+\beta)^{2}+16.
\]
   Put \(u:=\alpha+\beta\). Then
$$\alpha+\beta < \frac{d^{2}}{16}\Longleftrightarrow u^{4}-8u^{2}-16u+16> 0.$$

 Since \(\alpha\geq 2\) and \(\beta \geq 2\), it follows that
   \(u\geq 3\), whence \(\phi(u)= u^{4}-8u^{2}-16u+16\) is positive.
   Thus we get \(\alpha+\beta < \frac{d^{2}}{2^4}\), and putting \(v:=\alpha+\beta\) and \(\delta:=d\) we conclude
   that \(\epsilon_{1}^{2}\) cannot be a fifth power in
   \({k}_{1}\) either.

\end{proof}

%
%
%
%

\begin{corollary}
\label{cor:InfiniteFamily}
Let 
$${M}=\mathbb{Q}\left((\zeta_{5}-\zeta_{5}^{-1})\sqrt{(\alpha+\beta)^{2}-4}\right),$$
 where the integers \(\alpha,\,\beta\) satisfy the conditions
\eqref{eqn:Parameters}.
Let 
$\varphi$ denote the generator of
\(\mathrm{Gal}\left(\mathbb{Q}(\sqrt{5})/\mathbb{Q}\right)\).
Assume that the \(5\)-class group \(\mathrm{Cl}_{5}(M)\) of \(M\) is
of type \((5,5)\). Then \(M_5^{(1)}/ M\) contains six unramified
cyclic quintic extensions \(E_i/M\), which give rise to absolute extensions of degree \(20\) over \(\mathbb{Q}\), ordered the following way:
\begin{itemize}
 \item
\({E}_{1}={K}_{\epsilon_{1}^{2}}\) of Type (I) with
\(\mathrm{Gal}({E}_1/\mathbb{Q})\simeq {F}_{5,2}\), the splitting field of the polynomial \(f(X,\epsilon_{1}^{2})=X^{5}-5X^{3}+5X-(d+2)\); 
\item
\({E}_{2}={K}_{\epsilon_{2}}\) of Type (II) with \(\mathrm{Gal}({E}_2/\mathbb{Q})\simeq {F}_{5,3}\),
the splitting field of the polynomial \(f(X,\epsilon_{2})=X^{5}-5X^{3}+5X-(\alpha+5^{3}\beta)\);
\item
the other four extensions \(E_3,\, E_4=E_3^\varphi,\, E_5,\, E_6=E_5^\varphi\), which are non-Galois of Type (III) over \(\mathbb{Q}\) and form two conjugate pairs.
\end{itemize}
\end{corollary}

\begin{proof}
The claims are a consequence of Proposition
 \ref{prp:InfiniteFamily},
the ormulas of 
\eqref{eqn:KishiUnits}
 and the fact that
 \(\mathrm{Tr}_{\mathbb{Q}(\sqrt{d})}\left(\epsilon_{1}^{2}\right)=d+2\)
 and
\(\mathrm{Tr}_{\mathbb{Q}(\sqrt{5d})}\left(\epsilon_{2}\right)=\alpha+5^3\beta\).
\end{proof}

%
%
%
%

\begin{remark}
\label{rmk:Primary}
For \(d>0\), if the fundamental units of the real quadratic fields \(k_{i}\), \(i=1,2\), are \(5\)-primary,
then the field \(M\) has a non-Galois unramified cyclic quintic extension of Type (III)
\cite{Ki}.
In this case, there are four pairwise conjugate extensions of Type (III),
and among the remaining two Frobenius extensions one is of Type (I) and one is of Type (II)
\cite{Ki}.
Note that in the case \(d>0\) the cyclic quartic field \(M\) is imaginary.
Also, if \(5\) divides the class number of \(k_{i}\), \(i=1,2\),
there exists at most one \(5\)-primary element of \(k_{i}\), \(i=1,2\),
which gives rise to the Frobenius extensions of Type (I) and Type (II).
\end{remark}

%
%
%
%
%
%
%
%
%
%
%
%

\subsection{Real cyclic quartic fields \(M\) with \(d<0\)}
\label{ss:RealTheory}
\noindent
As before,
the two Frobenius groups \(F_{5,w}\) of order \(20\) with primitive root \(w\in\{2,3\}\) modulo \(5\) will be denoted
as in formula
\eqref{eqn:5Frobenius}.

%
%
%
%

\begin{proposition}
\label{prp:FrobeniusFieldsReal}
Let \({E}_1,\ldots,{E}_6\) be the six
unramified cyclic quintic extensions of the real cyclic quartic
field \({M}=\mathbb{Q}\left((\zeta_5-\zeta_5^{-1})\sqrt{d}\right)\),
\(d<0\), with \(5\)-class group \(\mathrm{Cl}_5(M)\simeq C_5\times
C_5\) or, more generally, of \(5\)-class rank \(2\). The properties
of these fields as absolute extensions \({E}_i/\mathbb{Q}\), in
dependence on the five cases in Table
 \ref{tbl:5RankCasesReal},
  are given as follows:
\begin{enumerate}
\item
In case \((\mathrm{a})\), all six fields \({E}_1,\ldots,{E}_6\) are
normal and share isomorphic automorphism groups
\(\mathrm{Gal}({E}_i/\mathbb{Q})\simeq {F}_{5,2}\), for \(1\le i\le6\).
\item
In case \((\mathrm{b})\), all six fields \({E}_1,\ldots,{E}_6\) are
normal and share isomorphic automorphism groups
\(\mathrm{Gal}({E}_i/\mathbb{Q})\simeq {F}_{5,3}\), for \(1\le i\le6\).
\item
In all the other cases \((\mathrm{c})\), \((\mathrm{d})\),
\((\mathrm{e})\), two extensions are normal with non-isomorphic
automorphism groups, say \(\mathrm{Gal}({E}_1/\mathbb{Q})\simeq
{F}_{5,2}\) and \(\mathrm{Gal}({E}_2/\mathbb{Q})\simeq {F}_{5,3}\),
but the other four extensions are non-Galois and form two conjugate
pairs \({E}_3\simeq {E}_4\) and \({E}_5\simeq {E}_6\).
\end{enumerate}
\end{proposition}

\renewcommand{\arraystretch}{1.0}

\begin{table}[ht]
\caption{Some possible \(5\)-class ranks \(r_1:=\varrho_5(k_1)\), \(r_2:=\varrho_5(k_2)\)
and invariants \(\delta_1\), \(\delta_2\) for the associated quadratic fields \(k_1,k_2\)
which are \(5\)-dual to a real cyclic quartic field \(M=\mathbb{Q}((\zeta_5-\zeta_5^{-1})\sqrt{d})\), \(d<0\),
with \(5\)-class rank \(r:=\varrho_5(M)=2:\)}
\label{tbl:5RankCasesReal}

\begin{center}
\begin{tabular}{|c||c|c|c|c|c|c|c|}
\hline
 Case & \(r_1\) & \(\delta_1\) & \(r_2\) & \(\delta_2\) \\
\hline
  (a) &  \(2\) &         \(0\) &   \(0\) &        \(0\) \\
  (b) &  \(0\) &         \(0\) &   \(2\) &        \(0\) \\
  (c) &  \(1\) &         \(0\) &   \(1\) &        \(0\) \\
  (d) &  \(2\) &         \(1\) &   \(1\) &        \(0\) \\
  (e) &  \(1\) &         \(0\) &   \(2\) &        \(1\) \\
\hline
\end{tabular}
\end{center}
\end{table}


\begin{proof}
Similar to the proof of Proposition \ref{prp:FrobeniusFieldsImag}.
\end{proof}

%
%
%
%
%
%
%
%
%
%
%
%
\vskip 4mm 

\section{The second \(5\)-class group \(\mathrm{G}_{5}^{(2)}M\) of \(M\)}
\label{s:Metabelianization}
\noindent
Based on the class number formula
\cite{Lm}
for dihedral relative extensions \({E}\) of
degree \(10\) over a base field \({F}\) with class number coprime to
\(5\), we are now in a  position to determine the isomorphism type
of the Galois group \(G_{5}^{(2)}{M} = \mathrm{Gal}(M_{5}^{(2)}/{M})\)
of the second Hilbert \(5\)-class field \({M}_{5}^{(2)}\) of a
cyclic quartic field
\({M}=\mathbb{Q}\left((\zeta_{5}-\zeta_{5}^{-1})\sqrt{d}\right)\)
with \(5\)-class group of type \((5,5)\), because its
unramified cyclic quintic extensions \({E}_{i},\,\, 1 \leq i \leq 6\),
turn out to be relatively dihedral  over the  quadratic subfield \({k}_0^+ =
\mathbb{Q}(\sqrt{5})\) of \({M}\), which has class number \(1\).

%
%
%
%

\begin{theorem}
\label{thm:ClassNumbers}
 The relation between the \(5\)-class numbers
\(\mathrm{h}_5({E}_i)\) of the six unramified cyclic quintic
extensions \({E}_i\), \(1\le i\le 6\), of \({M}\) and the \(5\)-class numbers
\(\mathrm{h}_5({L}_i)\) of their non-Galois subfields \({L}_i\), which
are of relative degree \(5\) over the  field
\({k}_0^+=\mathbb{Q}(\sqrt{5})\), is given by
\begin{equation}
\label{eqn:ClNrFormula}
\mathrm{h}_5({E}_i)=
\begin{cases}
\mathrm{h}_5({L}_i)^2     & \text{ if } \;\#\ker(j_{{E}_i/{M}})=25,\;\;
(U_{{k}_0^+}:\mathrm{N}_{{L}_i/{k}_0^+}(U_{{L}_i}))=1,
 \\
5\cdot\mathrm{h}_5({L}_i)^2 & \text{ if }\;  \#\ker(j_{{E}_i/{M}})=25,\;\;
(U_{{k}_0^+}:\mathrm{N}_{{L}_i/{k}_0^+}(U_{{L}_i}))=5,
 \\
5\cdot\mathrm{h}_5({L}_i)^2 & \text{ if }\;  \#\ker(j_{{E}_i/{M}})=5,\quad
(U_{{k}_0^+}:\mathrm{N}_{{L}_i/{k}_0^+}(U_{{L}_i}))=1,
 \\
25\cdot\mathrm{h}_5({L}_i)^2&\text{ if }\; \#\ker(j_{{E}_i/{M}})=5,\quad
(U_{{k}_0^+}:\mathrm{N}_{{L}_i/{k}_0^+}(U_{{L}_i}))=5,
\end{cases}
\end{equation}
where \(U_{F}\) denotes the unit group of a field \(F\).
\end{theorem}

\begin{proof}
According to Lemmermeyer
 \cite[eq. (5.2), p. 685]{Lm},
 we have the class number relation
\[
\mathrm{h}_5({E}_i)
=
\frac{(U_{{k}_0^+}:\mathrm{N}_{{L}_i/{k}_0^+}(U_{{L}_i}))}{\#\ker(j_{{E}_i/{M}})}\cdot\mathrm{h}_5({M})\cdot\mathrm{h}_5({L}_i)^2,
\]
  where \(\mathrm{h}_5({M})=25\), due to our general assumption on \({M}\).
Distinction between total principalization, \(\#\ker(j_{{E}_i/{M}})=25\), and partial principalization, \(\#\ker(j_{{E}_i/{M}})=5\),
immediately yields the four claimed cases, in dependence on the
unit norm indices \(u_i:=(U_{{k}_0^+}:\mathrm{N}_{{L}_i/{k}_0^+}(U_{{L}_i}))\).
\end{proof}

%
%
%
%

\begin{remark}
\label{rmk:UnitIndices}
In order to prove Theorem \ref{thm:ClassNumbers}
in a different manner,
we can use the class number formula, due to Lemmermeyer
\cite[Th.\,2.4, p.\,681]{Lm},
and the following Lemma
\ref{lem:UnitIndices}.
\end{remark}

%
%
%
%

\begin{lemma}
\label{lem:UnitIndices}
Let \(p\) be an odd prime and let \(F\) be a number field with class number coprime to \(p\).
Let \(k\) be a quadratic extension of \(F\).
Assume that \(L\) is an unramified cyclic extension of \(k\) of degree \(p\).
Then the extension \(L/F\) is Galois, dihedral of degree \(2p\), and we have the formula
\[a:=\left(U_{L}:U_{K}U_{K'}U_{k}\right)=
\frac{\left(U_{k}:U_{k}^{p}\right)\left(U_{F}:\mathrm{N}_{K/F}(U_{K})\right)}{\left(U_{F}:U_{F}^{p}\right)\left(U_{k}:\mathrm{N}_{L/k}(U_{L})\right)}\]
for the subfield unit index \(a\), where \(K\ne K^\prime\) denote two conjugate non-Galois subfields of \(L\).
\end{lemma}

Since \(p\ge 3\) is an odd prime
and the existence of an unramified cyclic extension \(L/k\) of degree \(p\) excludes
the irregular case \(p=3\), \(F=\mathbb{Q}\), \(k=\mathbb{Q}(\sqrt{-3})\) with \(h_k=1\),
either both fields \(k\) and \(F\) contain the \(p\)th roots of unity or both not.
Therefore, 
$$\frac{\left(U_{k}:U_{k}^{p}\right)}{\left(U_{F}:U_{F}^{p}\right)}=p^{r(k)-r(F)}$$
with the torsion-free Dirichlet unit ranks \(r(k)\) of \(k\) and \(r(F)\) of \(F\).
For an unramified extension \(L/k\),
the Theorem on the Herbrand quotient of \(U_L\) is equivalent with
\(\#\ker(j_{L/k})=p\cdot b\) with \(b:=\left(U_{k}:\mathrm{N}_{L/k}(U_{L})\right)\).
Using Lemma
\ref{lem:UnitIndices},
which can be found in
\cite[p.\,686]{Lm},
we can express the factor on the right hand side of the class number relation
\cite[Th.\,2.4, p.\,681]{Lm},
\[\mathrm{h}_p(L)=\frac{a}{p^{1+r(k)-r(F)}}\cdot\mathrm{h}_p(k)\cdot\mathrm{h}_p(K)^2,\]
in the form
$$\frac{a}{p^{1+r(k)-r(F)}}=\frac{a\cdot\left(U_{F}:U_{F}^{p}\right)}{p\cdot\left(U_{k}:U_{k}^{p}\right)}
=\frac{\left(U_{F}:\mathrm{N}_{K/F}(U_{K})\right)}{\#\ker(j_{L/k})},$$
which we have used for \(p=5\), \(F=k_0^+\), \(k=M\), \(L=E_i\), \(K=L_i\) in the proof of Theorem
\ref{thm:ClassNumbers}.

%
%
%
%
%
%
%
%
%
%
%
%

\subsection{Imaginary cyclic quartic fields \(M\) with \(d>0\)}
\label{ss:ImaginaryFields}


\begin{theorem}
\label{thm:ClassNumbersImag}
The \(5\)-class field tower of \(M\) has length \(\ell_5{M}=1\)
if and only if the second \(5\)-class group \(\mathrm{G}_5^{2}M\) of \(M\)
is the abelian \(5\)-group \(\langle 25,2\rangle\) of type \((5,5)\).
In this case,
\begin{enumerate}
\item
the \(5\)-class groups \(\mathrm{Cl}_5(E_i)\) are cyclic of order \(5\), for \(1\le i\le 6\),
\item
the \(5\)-class groups \(\mathrm{Cl}_5(L_i)\) are trivial, for \(1\le i\le 6\),
\item
the \(5\)-principalization of \(M\) is of type \(\mathrm{a}.1\), \(\varkappa(M)=(000000)\).
\end{enumerate}
\end{theorem}

\begin{proof}
For \(\mathrm{G}_5^{2}M\simeq\langle 25,2\rangle\),
we have the cyclic \(5\)-class groups \(\mathrm{Cl}_5(E_i)\simeq C_5\) and
six total principalizations \(\#\ker(j_{E_i/M})=25\).
According to Theorem
\ref{thm:ClassNumbers},
we obtain
\(\mathrm{h}_5(E_i)=5=u_i\cdot\mathrm{h}_5(L_i)^2\),
which enforces
\(\mathrm{h}_5(L_i)=1\) and \(u_i=5\), for all \(1\le i\le 6\).
\end{proof}

%
%
%
%

\begin{example}
\label{exm:AbelianImag}
The  values \(d=4357\) and \(d=4444\) give rise
to fields \({M}=\mathbb{Q}\left((\zeta_{5}-\zeta_{5}^{-1})\sqrt{d}\right)\)
with \(5\)-class group of type \((5,5)\) having a
single-stage \(5\)-class tower. Fields of this type are extremely rare, since
they form a fraction of \(\frac{2}{83}\) among the fields with \(0<d< 10 000\).  Therefore, only  about \(2\%\)
of the cases possess a single-stage tower.
\end{example}

%
%
%
%

\begin{proposition}
\label{prp:UnitIndicesImag}
Let \(M=\mathbb{Q}\left((\zeta_{5}-\zeta_{5}^{-1})\sqrt{d}\right)\), with \(d>0\),  be
an imaginary cyclic quartic field with \(5\)-class group of type \((5,5)\).
Let \(E_i,\,\,1\leq i\leq 6\), be the six unramified cyclic quintic extensions of \(M\)
and \(L_i\) their non-Galois subfields
of relative degree \(5\) over the field \(k_0^+=\mathbb{Q}(\sqrt{5})\).
Then the following holds true for each \(1\le i\le 6\):
\begin{enumerate}
\item
the subfield unit indices \(a_{i}:=\left(U_{E_{i}}:U_{L_{i}}U_{L_{i}^\prime}U_{M}\right)\) are equal to \(1\),
\item
the unit norm indices \(u_{i}\) satisfy the equivalence \(u_{i}=1\Longleftrightarrow\#\ker(j_{E_{i}/M})=5\), 
\item
the relations between the \(5\)-class numbers \(\mathrm{h}_5(E_i)\) and \(\mathrm{h}_5(L_i)\)
are given by
\[\mathrm{h}_5(E_i)=5\cdot \mathrm{h}_5(L_i)^{2}.\]
\end{enumerate}
\end{proposition}

\begin{proof}
According to Lemma
\ref{lem:UnitIndices},
we can deduce that
\[a_{i}b_{i}=\frac{\left(U_{M}:U_{M}^{5}\right)\left(U_{k_0^+}:\mathrm{N}_{L_{i}/k_0^+}(U_{L_{i}})\right)}
{\left(U_{k_0^+}:U_{k_0^+}^{5}\right)},\]
where \(b_{i}\) denotes the unit norm index \(\left(U_{M}:\mathrm{N}_{E_{i}/M}(U_{E_{i}})\right)\).
Since \(d>0\), the field \(M\) is imaginary and
it is a CM-field with maximal real subfield \(M^+=k_0^+\).
Hence, the torsion-free Dirichlet unit rank of \(M\) is \(r(M)=1\), and
\(U_{M}=\langle -1,\epsilon_{5}\rangle\),
where \(\epsilon_{5}\) denotes the fundamental unit of the quadratic field \(k_0^+=\mathbb{Q}(\sqrt{5})\).
This implies that
\[\left(U_{M}:U_{M}^{5}\right)=\left(U_{k_0^+}:U_{k_0^+}^{5}\right) \,\,
\text{ and }\,\, a_{i}b_{i}=u_{i}.\]

(1) To prove the first assertion,  it suffices to show the following equivalence:
\[u_{i}=1 \,\,\,\text{if and only if }\,\,\,\, b_{i}=1.\]
So it suffices to show that the fundamental unit \(\epsilon_{5}\) of \(k_0^+\),
which is also the fundamental unit of \(M\),
is the norm of a unit of \(E_{i}\) if and only if it is the norm of a unit of \(L_{i}\).
If \(u_{i}=1\), for \(i\in\lbrace 1,\ldots,6\rbrace\),
then \(\epsilon_{5}\) is the norm of a unit of \(L_{i}\) (a non-Galois subfield of \(E_{i}\)),
hence it is also the norm of the same unit in \(E_{i}\), and \(b_i=1\).

Now suppose that \(b_{i}=1\), for \(1\leq i\leq 6\).
Then there exists a unit \(\xi\in U_{E_{i}}\) such that
\mbox{\(\epsilon_{5}=\mathrm{N}_{E_{i}/M}(\xi)\)}, and we obtain the following chain of implications: 
$$ \begin{array}{ll}
\mathrm{N}_{M/\mathbb{Q}(\sqrt{5})}(\epsilon_{5})
&=
\mathrm{N}_{M/ \mathbb{Q}(\sqrt{5})}\left(\mathrm{N}_{E_{i}/ M}(\xi)\right)
\\[2mm]
&\qquad\qquad\Rightarrow
\epsilon_{5}^{2}
=
\mathrm{N}_{M/ \mathbb{Q}(\sqrt{5})}\left(\mathrm{N}_{E_{i}/ M}(\xi)\right)
=
\mathrm{N}_{L_{i}/ \mathbb{Q}(\sqrt{5})}\left(\mathrm{N}_{E_{i}/ L_{i}}(\xi)\right)
\\[2mm]
&\qquad\qquad\Rightarrow
\epsilon_{5}^{6}
=
\mathrm{N}_{L_{i}/ \mathbb{Q}(\sqrt{5})}\left(\mathrm{N}_{E_{i}/ L_{i}}(\xi^{3})\right)
\\[2mm]
&\qquad\qquad\Rightarrow
\epsilon_{5}\cdot\mathrm{N}_{L_{i}/ \mathbb{Q}(\sqrt{5})}(\epsilon_{5})
=
\mathrm{N}_{L_{i}/ \mathbb{Q}(\sqrt{5})}\left(\mathrm{N}_{E_{i}/ L_{i}}(\xi^{3})\right),
\end{array}
$$
whence
\[
\epsilon_{5}
=
\mathrm{N}_{L_{i}/ \mathbb{Q}(\sqrt{5})}\left(\epsilon_{5}^{-1}\cdot\mathrm{N}_{E_{i}/ L_{i}}(\xi^{3})\right).
\]
Since the element \(\epsilon_{5}^{-1}\cdot\mathrm{N}_{E_{i}/ L_{i}}(\xi^{3})\) is a unit of \(L_{i}\),
we obtain the index \(u_{i}=1\).
On the other hand, the possible values of \(b_{i}\) and \(u_{i}\) are \(\lbrace 1,5\rbrace\),
and we can deduce that \(u_{i}=b_{i}\).
Finally, it follows from the equation \(a_{i}b_{i}=u_{i}\) that \(a_{i}=1\).

(2) The result follows immediately from the fact that \(\#\ker(j_{E_{i}/ M})=5\cdot b_{i}\).

(3) According to Theorem
\ref{thm:ClassNumbers},
we have two possible cases,
\[
u_{i}=1 \,\, \text{and}\,\,\# \ker j_{E_{i}/ M}=5,
\]
and
\[
u_{i}=5 \,\, \text{and}\,\,\# \ker j_{E_{i}/ M}=25.
\]
In both cases, the class number formula is given by
\(\mathrm{h}_5(E_i)=5\cdot\mathrm{h}_5(L_i)^{2}\).
\end{proof}

%
%
%
%

\begin{theorem}
\label{thm:NonAbelianImag}
Let 
\(M=\mathbb{Q}\left((\zeta_{5}-\zeta_{5}^{-1})\sqrt{d}\right)\) 
with 
\(d>0\) 
be an imaginary cyclic quartic field 
with \(5\)-class group
\(\mathrm{Cl}_{5}(M) \simeq C_{5} \times C_{5}\). 
If the second \(5\)-class group 
\(G:=\mathrm{G}_5^2(M)\) 
of 
\(M\) 
is non-abelian,
then the coclass \(\mathrm{cc}(G)\) of \(G\) is greater than or equal to \(2\),
\(\mathrm{cc}(G)\geq 2\).
\end{theorem}

\begin{proof}
Assume that \(G\) is non-abelian of coclass \(cc(G)=1\).  Then the
possible capitulation types of \(M\) in the six intermediary cyclic
quintic extensions of \({M}_{5}^{(2)}/ {M}\), noted by
$$
{E}_{1},\,{E}_{2},\,{E}_{3},\,{E}_{4}={E}_{3}^\varphi,\,{E}_{5},\,{E}_{6}={E}_{5}^\varphi,
$$
are given by \(\varkappa(G)=(111111)\) or \(\varkappa(G)=(\ell 00000),\,\, \ell\in\{0,1,2\}\).

First we consider the type \(\varkappa(G)=(111111)\). In this case,
the group \(G\) is the extra special \(5\)-group of order \(5^{3}\) and exponent \(5^2\),
 whose maximal normal subgroups are of order \(5^{2}\). This implies that the
\(5\)-class number of each  \({E}_{i}\) is equal to \(5^{2}\). Using
Proposition
 \ref{prp:UnitIndicesImag}, however,
  we conclude that the valuation
  \(v_{5}(h_{5}({E}_{i}))\)
of the \(5\)-class number of \({E}_{i}\) must be
odd, which is a contradiction. Thus  the type
\(\varkappa(G)=(111111)\) cannot occur.

For the three other types, we have total capitulation in the five
extensions
$$
{E}_{2},\,{E}_{3},\,{E}_{4}={E}_{3}^\varphi,\,{E}_{5},\,{E}_{6}={E}_{5}^\varphi,
$$
so the value of the index \(b_{i}\), \(2\le i\le 6\),
is \(b_{i}=5\), whence \(u_{i}=5\). On the other hand, for
\(2\le i\le 6\) we again have \(\mathrm{h}_{5}({E}_{i})=5^{2}\), which is a contradiction,
since by Proposition
\ref{prp:UnitIndicesImag},
the valuation
\(v_{5}(h_{5}({E}_{i}))\)
must be odd.
\end{proof}

%
%
%
%

\begin{proposition}[{\sf Number of fields}]
\label{prp:CountImag}
In the range \(0<d<10000\) of fundamental discriminants \(d\)
of real quadratic fields \({k}_1=\mathbb{Q}(\sqrt{d})\) with \(\gcd(5,d)=1\),
there exist precisely \(\mathbf{83}\) cases such that
the \(5\)-dual field \({M}=\mathbb{Q}((\zeta_5-\zeta_5^{-1})\sqrt{d})\) of \({k}_1\)
has a \(5\)-class group \(\mathrm{Cl}_5(M)\) of type \((5,5)\).
\end{proposition}

\begin{proof}
See Tables
\ref{tbl:CycQrtFld1}
and
\ref{tbl:CycQrtFld2}.
\end{proof}

%
%
%
%

\begin{theorem}[{\sf Two-stage towers of \(5\)-class fields with Schur \(\sigma\)-groups}]
\label{thm:SchurSigmaImag}
\mbox{}
 
 $(1)$
If the \(5\)-dual field \(M\) of \(k_1\) has
\(5\)-principalization type \(\varkappa(M)=(125643)\) with two fixed points and a \(4\)-cycle,
then the abelian type invariants of \(E_1,\ldots,E_6\) are \(\tau(M)=\lbrack (1^3)^2,(21)^4\rbrack\),
and the \(5\)-class tower group is the Schur \(\sigma\)-group
\(\mathrm{G}_5^{(\infty)}{M}=\mathrm{G}_5^{(2)}{M}\simeq\langle 5^5,\mathbf{11}\rangle\). 

$(2)$
If the \(5\)-dual field \(M\) of \(k_1\) has
\(5\)-principalization type \(\varkappa(M)=(123456)\), the identity permutation,
then the abelian type invariants of \(E_1,\ldots,E_6\) are \(\tau(M)=\lbrack (1^3)^6\rbrack\),
and the \(5\)-class tower group is the Schur \(\sigma\)-group
\(\mathrm{G}_5^{(\infty)}{M}=\mathrm{G}_5^{(2)}{M}\simeq\langle 5^5,\mathbf{14}\rangle\). \\
 
\end{theorem}

\begin{proof}
In each case, the length of the \(5\)-class tower of \(M\) is given by \(\ell_5(M)=2\),
since \(G:=\mathrm{G}_5^{(2)}{M}\) is a Schur \(\sigma\)-group with balanced presentation,
i.e., relation rank \(d_2(G)=d_1(G)=\varrho_5(M)=2\).
\end{proof}

Examples for part  (1) are the  \(\mathbf{23}\) (about \(\mathbf{28}\%\)) real quadratic fields \(k_1\) starting with the following discriminants:\\
\centerline{$d\in\lbrace 457,\ 501,\ 1996,\ 2573,\ 3253,\ 4189,\ 4957,\ 5129,\ 5233,\ 5308,\ 5361,\ \ldots\rbrace.$}

Examples for  part (2) are the  \(\mathbf{11}\) (about \(\mathbf{13}\%\)) real quadratic fields \(k_1\) with the following discriminants:\\
\centerline{$d\in\lbrace 581,\ 753,\ 2296,\ 2829,\ 4553,\ 5116,\ 5736,\ 6761,\ 7489,\ 9013,\ 9829\rbrace,$}
verifying a conjecture by O. Taussky in
\cite{Ts},
and announced in
\cite[\S 3.5.2, p.\,448]{Ma4},
except \(2829\).

%
%
%
%

\begin{remark}
\label{rmk:SchurSigma}
The pairs of conjugate non-Galois extensions \(E_3\simeq E_4\) and \(E_5\simeq E_6\) of \(M\)
are not adjacent in the factor \((3546)\) of the cycle pattern \((1)(2)(3546)\)
of the \(4\)-cycle \(\varkappa(M)=(125643)\),
and the Frobenius extensions \(E_1,E_2\) correspond to the fixed points \((1),(2)\).
The identity \(\varkappa(M)=(123456)\), which does not have two distinguished fixed points a priori,
is endowed with a random arithmetical bipolarization by the two Frobenius extensions \(E_1,E_2\).
\end{remark}

%
%
%
%
%
%
%
%
%
%
%
%
%
%
%
%
%
%
%

\bigskip
\noindent
Figure
\ref{fig:5MirrSbfldLatt}
visualizes the situation of a two-stage \(5\)-class tower in the Theorems
\ref{thm:SchurSigmaImag},
\ref{thm:UnusualSigmaImag},
\ref{thm:SchurSigmaReal}.

\begin{figure}[ht]
\caption{{\sf The \(5\)-class tower \({M}_5^{(\infty)}\) of 
 \({M}=\mathbb{Q}\left((\zeta_5-\zeta_5^{-1})\sqrt{d}\right)\) when   \(\#\mathrm{G}_5^{(2)}{{M}}=5^5\)}}
\label{fig:5MirrSbfldLatt}



\setlength{\unitlength}{1cm}
\begin{picture}(14,14)(-8,0)


\put(-8,13.3){\makebox(0,0)[cb]{Degree}}
\put(-8,12){\vector(0,1){1}}
\put(-8,12){\line(0,-1){12}}
\put(-8.1,12){\line(1,0){0.2}}
\put(-8.2,12){\makebox(0,0)[rc]{\(12\,500\)}}
\put(-8.1,10){\line(1,0){0.2}}
\put(-8.2,10){\makebox(0,0)[rc]{\(100\)}}
\put(-8.1,7){\line(1,0){0.2}}
\put(-8.2,7){\makebox(0,0)[rc]{\(20\)}}
\put(-8.1,6){\line(1,0){0.2}}
\put(-8.2,6){\makebox(0,0)[rc]{\(8\)}}
\put(-8.1,4){\line(1,0){0.2}}
\put(-8.2,4){\makebox(0,0)[rc]{\(4\)}}
\put(-8.1,2){\line(1,0){0.2}}
\put(-8.2,2){\makebox(0,0)[rc]{\(2\)}}
\put(-8.1,0){\line(1,0){0.2}}
\put(-8.2,0){\makebox(0,0)[rc]{\(1\)}}

\put(-3,0){\circle*{0.2}}
\multiput(-6,2)(3,0){2}{\circle*{0.2}}

\put(0,2){\circle*{0.2}}
\multiput(-3,4)(3,0){3}{\circle*{0.2}}
\put(0,6){\circle*{0.2}}

\multiput(-5,7)(10,0){2}{\circle{0.2}}
\multiput(-3,7)(2,0){4}{\circle{0.1}}
\put(0,10){\circle{0.2}}
\put(0,12){\circle{0.2}}

\multiput(0,10)(-5,-3){2}{\line(5,-3){5}}
\multiput(0,10)(5,-3){2}{\line(-5,-3){5}}

\put(-3,0){\line(-3,2){3}}
\multiput(-3,0)(0,2){2}{\line(0,1){2}}
\multiput(-3,0)(-3,2){2}{\line(3,2){3}}

\multiput(0,2)(3,2){2}{\line(-3,2){3}}
\multiput(0,2)(0,2){2}{\line(0,1){2}}
\multiput(0,2)(-3,2){2}{\line(3,2){3}}

\multiput(0,4)(-1,3){2}{\line(1,3){1}}
\multiput(0,4)(1,3){2}{\line(-1,3){1}}
\multiput(0,4)(-3,3){2}{\line(1,1){3}}
\multiput(0,4)(3,3){2}{\line(-1,1){3}}
\multiput(0,4)(-5,3){2}{\line(5,3){5}}
\multiput(0,4)(5,3){2}{\line(-5,3){5}}

\put(0,10){\line(0,1){2}}

\put(-3,-0.2){\makebox(0,0)[ct]{\(\mathbb{Q}\)}}
\put(-5.7,2){\makebox(0,0)[lc]{\(k_1=\mathbb{Q}\left(\sqrt{d}\right)\)}}
\put(-2.8,2){\makebox(0,0)[lc]{\(k_2=\mathbb{Q}\left(\sqrt{5d}\right)\)}}
\put(0.2,2){\makebox(0,0)[lc]{\(\mathbb{Q}\left(\sqrt{5}\right)=k_0^+\)}}
\put(0,3.9){\makebox(0,0)[ct]{\(M=\mathbb{Q}\left((\zeta_5-\zeta_5^{-1})\sqrt{d}\right)\)}}
\put(3.2,4){\makebox(0,0)[lc]{\(\mathbb{Q}\left(\zeta_5\right)=k_0\)}}

\put(-5.2,7){\makebox(0,0)[rc]{\(E_1\)}}
\put(-3.2,7){\makebox(0,0)[rc]{\(E_3\)}}
\put(-1.2,7){\makebox(0,0)[rc]{\(\simeq\quad E_4\)}}
\put(1.2,7){\makebox(0,0)[lc]{\(E_5\quad\simeq\)}}
\put(3.2,7){\makebox(0,0)[lc]{\(E_6\)}}
\put(5.2,7){\makebox(0,0)[lc]{\(E_2\)}}

\put(-0.2,10.2){\makebox(0,0)[rc]{\(M_5^{(1)}\)}}
\put(-0.2,12.4){\makebox(0,0)[rc]{\(M_5^{(2)}\)}}
\put(0,12.4){\makebox(0,0)[lc]{\(=M_5^{(\infty)}\)}}
\end{picture}
`\mbox{} \vskip 1cm 
\end{figure}
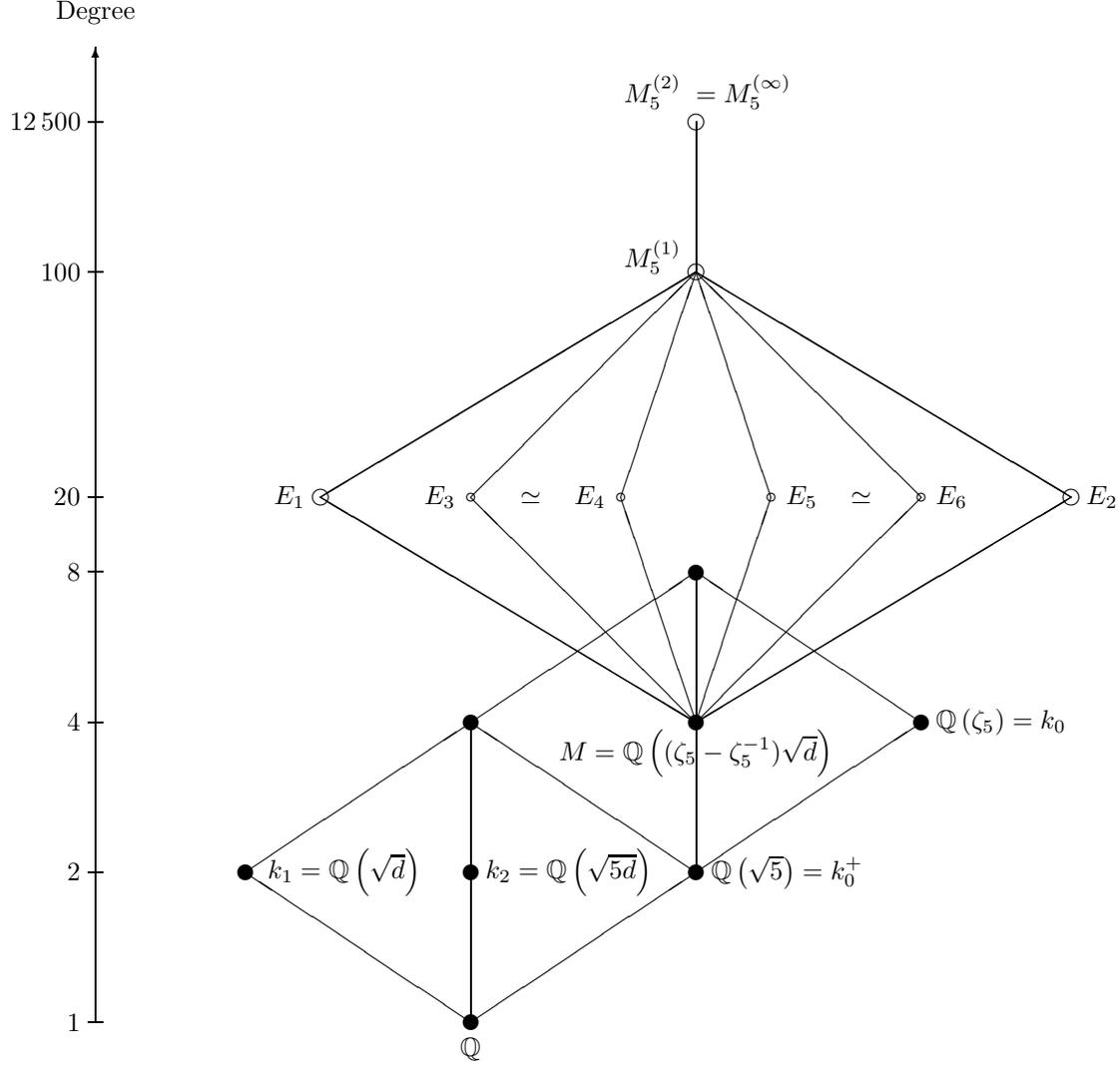 \vskip 0cm

%
%
%
%
%
%
%
%
%
%
%
%
%
%
%
%
%
%
%

\begin{theorem}[{\sf Two-stage towers of \(5\)-class fields with unusual capable weak \(\sigma\)-groups}]
\label{thm:UnusualSigmaImag}
  \mbox{}

$(1)$
If the \(5\)-dual field \(M\) of \(k_1\) has
\(5\)-principalization type \(\varkappa(M)=(022222)\), nearly constant with a single total capitulation and a single fixed point,
then the abelian type invariants of \(E_1,\ldots,E_6\) are \(\tau(M)=\lbrack (1^3)^2,(21)^4\rbrack\),
and the \(5\)-tower group is
\(\mathrm{G}_5^{(\infty)}{M}=\mathrm{G}_5^{(2)}{M}\simeq\langle 5^5,\mathbf{4}\rangle\).

$(2)$
If the \(5\)-dual field \(M\) of \(k_1\) has
\(5\)-principalization type \(\varkappa(M)=(124365)\) with two fixed points and two disjoint \(2\)-cycles,
then the abelian type invariants of \(E_1,\ldots,E_6\) are \(\tau(M)=\lbrack (1^3)^2,(21)^4\rbrack\),
and the \(5\)-class tower group is
\(\mathrm{G}_5^{(\infty)}{M}=\mathrm{G}_5^{(2)}{M}\simeq\langle 5^5,\mathbf{7}\rangle\). 
\end{theorem}
\begin{proof}
In each case, the length of the \(5\)-class tower of \(M\) is given by \(\ell_5(M)=2\),
since \(G:=\mathrm{G}_5^{(2)}{M}\) is a metabelian \(\sigma\)-group with trivial cover
\cite[Def.\,5.1, p.\,30]{Ma10},
according to Heider and Schmithals
\cite[p.\,20]{HeSm}.
The presentation of \(G\) is not balanced, since the relation rank \(d_2(G)=3\) is too big.
However, the Shafarevich Theorem
\cite{Sh},
in its corrected version
\cite[Th.\,5.1, p.°,28]{Ma10},
ensures that
\(d_2(G)\le d_1(G)+r=3\),
(it just reaches the admissible upper bound),
since the generator rank of \(G\) and the torsion-free Dirichlet unit rank of \(M\) with signature \((0,2)\)
are given by \(d_1(G)=\varrho_5(M)=2\) and \(r=0+2-1=1\).
The \(5\)-tower groups \(\langle 5^5,\mathbf{4}\rangle\) and \(\langle 5^5,\mathbf{7}\rangle\) are unusual,
because they are not strong \(\sigma\)-groups
and thus are forbidden for (imaginary and real) quadratic base fields
\cite{Sf}.
\end{proof}

Examples for part (1 ) are the \(\mathbf{22}\)  (about \(\mathbf{ 27}\%\)) real quadratic fields \(k_1\) starting with the following discriminants: \\
\centerline{\(d\in\lbrace 257,\ 764,\ 1708,\ 1853,\ 2008,\ 2189,\ 3129,\ 4504,\ 4861,\ 5241,\ 5269,\ \ldots\rbrace\).}

Examples for part (2) are  the  \(\mathbf{16}\) (about \(\mathbf{19}\%\)) real quadratic fields \(k_1\) starting with the following discriminants: \\
\centerline{\(d\in\lbrace 508,\ 509,\ 629,\ 881,\ 1113,\ 1192,\ 1704,\ 1829,\ 3121,\ 4461,\ 7032,\ \ldots\rbrace\).}

%
%
%
%

\begin{remark}
\label{rmk:UnusualSigma}
The pairs of conjugate non-Galois extensions \(E_3\simeq E_4\) and \(E_5\simeq E_6\) of \(M\)
correspond to the factors \((34)\) and \((56)\) of the cycle pattern \((1)(2)(34)(56)\)
of the two disjoint \(2\)-cycles \(\varkappa(M)=(124365)\),
and the Frobenius extensions \(E_1,E_2\) correspond to the fixed points \((1),(2)\).
For the nearly constant type \(\varkappa(M)=(022222)\),
the first (resp. second)  Frobenius extension \(E_1\) (resp. \(E_2\))
corresponds to the single total capitulation (resp. the single fixed point).
\end{remark}

%
%
%
%

\begin{theorem}[{\sf Single-stage towers of \(5\)-class fields with abelian group}]
\label{thm:AbelianImag}
For the \(\mathbf{2}\) (about \(\mathbf{2}\%\)) real quadratic fields \({k}_1\) with discriminants
\(d\in\lbrace 4357,\ 4444\rbrace\),
the \(5\)-dual field \({M}\) of \({k}_1\) has
\(5\)-principalization type \(\varkappa({M})=(000000)\), a constant with six total capitulations;
the abelian type invariants of \(E_1,\ldots,E_6\) are \(\tau({M})=\lbrack (1)^6\rbrack\),
and thus the \(5\)-class tower is abelian with group
\(\mathrm{G}_5^{(\infty)}{{M}}=\mathrm{G}_5^{(1)}{{M}}\simeq\langle 5^2,\mathbf{2}\rangle\) and length \(\ell_5(M)=1\).
\end{theorem}

\begin{proof}
Here, the \(5\)-class tower is abelian with length \(\ell_5(M)=1\), according to Theorem
\ref{thm:ClassNumbersImag}.
\end{proof}

%
%
%
%
\mbox{} \vskip 5 cm 

\begin{remark}
\label{rmk:AbelianImag}
Outside the range \(0<d<10^4\)
of our systematic investigations,
we have discovered three occurrences of case (g) in Table
\ref{tbl:5RankCasesImag}.
For the real quadratic fields \(k_1\) with discriminants
\(d\in\lbrace 244641,\ 1277996,\ 1915448\rbrace\)
the \(5\)-dual field \(M\) of \(k_1\) has
\(5\)-principalization type \(\varkappa(M)=(000000)\), a constant with six total capitulations,
abelian type invariants \(\tau(M)=\lbrack (1)^6\rbrack\),
and abelian \(5\)-class tower with group
\(G_5^{(\infty)}{M}=G_5^{(1)}{M}\simeq\langle 5^2,\mathbf{2}\rangle\) and length \(\ell_5(M)=1\).
The invariants are given by
\((r_1,r_2,\delta_1,\delta_2)=(2,0,1,1)\).

\end{remark}

%
%
%
%

\begin{theorem}[{\sf Frobenius and non-Galois extensions}]
\label{thm:FrobeniusImag}
The properties of the absolute extensions \({E}_i/\mathbb{Q}\)
and the values of the invariants in the Quintic Reflection Theorem, Table
\ref{tbl:5RankCasesImag},
and Proposition
\ref{prp:FrobeniusFieldsImag},
for the \(\mathbf{83}\) cases in Proposition
\ref{prp:CountImag}
are the following ones:
\begin{enumerate}
\item[(i)]
For the \(\mathbf{2}\) cases with \(\ell_5(M)=1\) in Theorem
\ref{thm:AbelianImag},
we have
\[
(r_1,r_2,\delta_1,\delta_2)=(1,0,0,1)\;\text{ and }\;
\mbox{\rm{Gal}}({E}_i/\mathbb{Q})\simeq {F}_{5,2}\text{ for }1\le i\le 6 \;\;(\text{Case }(\mathrm{a})
\]
\item[(ii)]
For the other \(\mathbf{81}\) cases,
including the \(\mathbf{34}\) cases of \(\ell_5(M)=2\) in Theorem
\ref{thm:SchurSigmaImag}
and the \(\mathbf{38}\) cases of \(\ell_5(M)=2\) in Theorem
\ref{thm:UnusualSigmaImag},
we have
pairwise conjugate non-Galois extensions 
$${E}_3\simeq {E}_4, \; {E}_5\simeq {E}_6  \; \mbox{ with }\; 
\mathrm{Gal}({E}_1/ \mathbb{Q})\simeq {F}_{5,2},\; \mathrm{Gal}({E}_2/ \mathbb{Q})\simeq {F}_{5,3}$$
and  
$$\hskip 2cm  \begin{cases}
(r_1,r_2,\delta_1,\delta_2)=(1,0,1,0),\text{ for }  d\in\lbrace 1996,\ 3121,\ 3129,\ 3253,\ 5241,\ 5269,\\
\hspace{5cm}\  5308,6113,\ 8309,\ 8689,\ 9829\rbrace\;(\text{Case }(\mathrm{e})),\\[2mm]
(r_1,r_2,\delta_1,\delta_2)=(0,1,0,1),\text{ for }  d\in\lbrace 5116,\ 8972,\ 9013\rbrace\;(\text{Case }(\mathrm{f})),\\[2mm]
(r_1,r_2,\delta_1,\delta_2)=(1,1,1,1),\text{ for }  d\in\lbrace 4504,\ 6949,\ 7221,\ 7229,\ 9669\rbrace\;(\text{Case }(\mathrm{c})),\\[2mm]
(r_1,r_2,\delta_1,\delta_2)=(0,0,0,0),\text{ otherwise}\;(\text{Case }(\mathrm{d})).
\end{cases}
$$
\end{enumerate}
\end{theorem}

\begin{proof}
See Tables
\ref{tbl:CycQrtFld1}
and
\ref{tbl:CycQrtFld2}.
\end{proof}

%
%
%
%
%
%
%
%
%
%
%
%

\subsection{Real cyclic quartic fields \(M\) with \(d<0\)}
\label{ss:RealFields}

%
%
%
%

\begin{proposition}
\label{prp:UnitIndicesReal}
Let \(M=\mathbb{Q}\left((\zeta_{5}-\zeta_{5}^{-1})\sqrt{d}\right)\) with \(d<0\) be
a real cyclic quartic field with 5-class group of type \((5,5)\).
Denote by \(E_i,\,\,1\leq i\leq 6\),
the six unramified cyclic quintic extensions of \(M\)
and by \(L_i\) their non-Galois subfields,
which are of relative degree \(5\) over the field \(k_0^+=\mathbb{Q}(\sqrt{5})\).

$(1)$
If \(\#ker(j_{E_{i}/ M})=5\), then the unit norm index is 
$$u_{i}=(U_{k_0^+}:\mathrm{N}_{L_i/ k_0^+}(U_{L_i}))=1,$$
and in this case the subfield unit index \(a_{i}=\left(U_{E_{i}}:U_{L_{i}}U_{L_{i}^\prime}U_{M}\right)\) is equal to \(25\).

$(2)$
The relation between the \(5\)-class numbers \(\mathrm{h}_5(E_i)\)
and \(\mathrm{h}_5(L_i)\) is given by
$$
\mathrm{h}_5(E_i)=
\begin{cases}
5\cdot\mathrm{h}_5(L_i)^{2} & \text{ if } b_{i}=u_{i}, \\
\mathrm{h}_5(L_i)^{2} & \text{ if } b_{i}\neq u_{i}, \text{ where } b_i=\left(U_{M}:\mathrm{N}_{E_{i}/M}(U_{E_{i}})\right).
\end{cases}
$$
\end{proposition}

\begin{proof}
Since \(d<0\), the field \(M\) is totally real,
and the Dirichlet rank of its torsion-free unit group is given by \(r(M)=3\).

(1)
Denote by \(U_{M/ \mathbb{Q}(\sqrt{5})}\) the group of relative units 
$$\lbrace\epsilon\in U_M\mid\mathrm{N}_{M/ \mathbb{Q}(\sqrt{5})}\left(\epsilon\right)=1\rbrace.$$
For a cyclic quartic field \(K/ \mathbb{Q}\) with real quadratic subfield \(k\),
Hasse showed that the group \(U_{k}U_{K/ k}\) has index at most \(2\)
in the full group of units \(U_{K}\).
In our case,
\(U_{\mathbb{Q}(\sqrt{5})}U_{M/ \mathbb{Q}(\sqrt{5})}\) has index at most \(2\) in \(U_{M}\),
where
$$U_{M}=\langle -1,\epsilon_{5},\eta,\eta^{\mu\tau}\rangle,$$
with \(\eta\) satisfying \(\eta^{1+\mu\tau}=\pm 1\).
If \(\#\ker(j_{E_{i}/ M})=5\),
which means that the unit norm index \(b_{i}=(U_{M}:\mathrm{N}_{E_i/ M}(U_{E_i}))\) is equal to \(1\),
then all units of \(M\) are the norms of a unit of \(E_{i}\),
in particular \(\epsilon_{5}\).
In the same manner as in the proof of claim 1 of Proposition
\ref{prp:UnitIndicesImag},
we deduce that \(\epsilon_{5}\) is also the norm of a unit of \(L_{i}\),
whence \(u_{i}=1\).
On the other hand, by applying Lemma
\ref{lem:UnitIndices},
we deduce that \(a_{i}\cdot b_{i}=25\cdot u_{i}\), and consequently
\(a_{i}=25\).

(2)
According to Theorem
\ref{thm:ClassNumbers}
or the Lemmermeyer class number formula
\cite[Th.\,2.4, p.\,681]{Lm},
we conclude that
\(\mathrm{h}_5(E_i)=5\cdot\mathrm{h}_5(L_i)^{2}\) if
\(b_{i}=1\) or \(\left(b_{i}=5 \,\, \text{and}\,\, u_{i}=5 \right)\).
But if \(b_{i}=5 \,\, \text{and}\,\, u_{i}=1\),
we have \(\mathrm{h}_5(E_i)=\mathrm{h}_5(L_i)^{2}\),
which completes the proof.
\end{proof}

%
%
%
%

\begin{remark}
\label{rmk:UnitIndicesReal}
For totally real or imaginary cyclic quartic fields \(M\),
the last case    of Theorem
\ref{thm:ClassNumbers}  given by  \mbox{\(\mathrm{h}_5(E_i)=25\cdot\mathrm{h}_5(L_i)^2\)}  
is impossible for any \( i\in   \{1,2,3,4,5,6\}\).
\end{remark}

%
%
%
%

\begin{proposition}
\label{prp:NonAbelianReal}
Let \(M=\mathbb{Q}\left((\zeta_{5}-\zeta_{5}^{-1})\sqrt{d}\right)\) with \(d<0\)
be a real cyclic quartic field with \mbox{ \(5\)-class} group of type \((5,5)\).
Let \(E_i,\,\,1\leq i\leq 6\), be the six unramified cyclic quintic extensions of \(M\).
Denote by \(G:=\mathrm{G}_{5}^{(2)}(M)\) the second \(5\)-class group of \(M\) and
assume that the order of \(G\) is equal to \(|G|=5^{3}\).
Then the transfer kernel type of \(G\) is \(\varkappa(G)=(000000)\)
(capitulation type of \(M\) in the six unramified extensions \(E_{i}\))
and the transfer target type of \(G\) is \(\tau(G)=\left\lbrack(1^2)^{6}\right\rbrack\).
\end{proposition}

\begin{proof}
In this case, the group \(G\) is extra special of maximal class.
Thus, the possible types of capitulation are \((111111)\) and \((000000)\).
First, we know that the type \((111111)\) is not possible,
because in this case \(\mathrm{h}_5(E_i)=5^2\) and \(b_i=1\),
which contradicts claim (2) of Proposition
\ref{prp:UnitIndicesReal}.
Thus, the transfer kernel type of \(G\) is \(\varkappa(G)=(000000)\).

On the other hand, for all \(1\leq i\leq 6\),
the unit norm index is \(b_{i}=5\), and \(u_{i}\) must be equal to \(1\).
Otherwise, the Lemmermeyer class number formula
\cite[Th.\,4.1, p.\,456]{Lm}
implies \(|G|\geq 5^{4}\).
Thus, for all \(1\leq i\leq 6\), we have
\(\mathrm{h}_5(E_i)=\mathrm{h}_5(L_i)^2\) and \(\mathrm{h}_5(L_i)=5\).
Since the six extensions \(E_{i}\) are of type \(A\) in the sense of Taussky
and \(\mathrm{h}_5(E_i)=5^{2}\),
we deduce that \(\mathrm{Cl}_{5}(E_{i})\) is of type \((5,5)\).
Thus \(\tau(G)=\left\lbrack(1^2)^{6}\right\rbrack\).
\end{proof}

%
%
%
%

\begin{remark}
\label{rmk:NonAbelianReal}
Assume that the group \(G\) is not abelian and \(\varkappa(G)=(000000)\).
Then the prime \(5\) must divide the class number of the fields \(L_{i}\).
Because, in this case
\(b_{i}=5\) and \(u_{i}=1\,\text{ or }\,5\).
The case \(u_{i}=1\) is obvious.
Now suppose that \(u_{i}=5\).
If \(5\) does not divide
\(\mathrm{h}(L_i)\), then \(\mathrm{h}_5(E_i)=5\),
Hence \(E_{i}\) is an unramified extension of \(M\) and
\(\mathrm{h}_5(E_i)=\frac{\mathrm{h}_5(M)}{5}\).
Then \(M_{5}^{(2)}=M_{5}^{(1)}\) and the group \(G\) is abelian,
which is a contradiction.
\end{remark}

%
%
%
%

\begin{proposition}[{\sf Number of fields}]
\label{prp:CountReal}
In the range \(-200000<d<0\) of fundamental discriminants \(d\)
of imaginary quadratic fields \({k}_1=\mathbb{Q}(\sqrt{d})\) with \(\gcd(5,d)=1\),
there exist precisely \(\mathbf{93}\) cases such that
the \(5\)-dual field \({M}=\mathbb{Q}\left((\zeta_5-\zeta_5^{-1})\sqrt{d}\right)\) of \({k}_1\)
has a \(5\)-class group \(\mathrm{Cl}_5(M)\) of type \((5,5)\).
\end{proposition}

\begin{proof}
See Tables
\ref{tbl:CycQrtFld3},
\ref{tbl:CycQrtFld4}
and
\ref{tbl:CycQrtFld5}.
\end{proof}

%
%
%
%

\begin{theorem}[{\sf Two-stage towers of \(5\)-class fields with groups of low order}]
\label{thm:LowSigmaReal} \mbox{}
 
$(1)$
If the \(5\)-dual field \({M}\) of \({k}_1\) has
\(5\)-principalization type \(\mathrm{a}.1\), \(\varkappa({M})=(000000)\), a constant with six total capitulations,
and the abelian type invariants of \(E_1,\ldots,E_6\) are \(\tau({M})=\lbrack (1^2)^6\rbrack\),
then the \(5\)-tower group is the extra special group
\(\mathrm{G}_5^{(\infty)}{{M}}=\mathrm{G}_5^{(2)}{{M}}\simeq\langle 5^3,\mathbf{3}\rangle\). 
 
$(2)$
If the \(5\)-dual field \({M}\) of \({k}_1\) has
\(5\)-principalization type \(\mathrm{a}.2\), \(\varkappa({M})=(100000)\) with a fixed point and five total capitulations,
and the abelian type invariants of \(E_1,\ldots,E_6\) are \(\tau({M})=\lbrack 1^3,(1^2)^5\rbrack\),
then the \(5\)-tower group is the Schur\(+1\) \(\sigma\)-group
\(\mathrm{G}_5^{(\infty)}{{M}}=\mathrm{G}_5^{(2)}{{M}}\simeq\langle 5^4,\mathbf{8}\rangle\).

$(3)$
If the \(5\)-dual field \({M}\) of \({k}_1\) has
\(5\)-principalization type \(\mathrm{a}.1\), \(\varkappa({M})=(000000)\), a constant with six total capitulations,
and the abelian type invariants of \(E_1,\ldots,E_6\) are \(\tau({M})=\lbrack 1^3,(1^2)^5\rbrack\),
then the \(5\)-tower group is the mainline group
\(\mathrm{G}_5^{(\infty)}{{M}}=\mathrm{G}_5^{(2)}{{M}}\simeq\langle 5^4,\mathbf{7}\rangle\). 

\end{theorem}

\begin{proof}
In each case, the length of the \(5\)-class tower of \(M\) is given by \(\ell_5(M)=2\),
according to Blackburn
\cite{Bl},
since \(G:=\mathrm{G}_5^{(2)}{M}\) is a \(\sigma\)-group
with at most two-generated commutator subgroup \(G^\prime\in\lbrace 1,1^2\rbrace\).
The presentation of \(G\) is not balanced, since the relation rank \(d_2(G)\in\lbrace 3,4\rbrace\) is too big.
However, the Shafarevich Theorem
\cite{Sh}
in its corrected version
\cite[Th.\,5.1, p.\,28]{Ma10}
ensures that
\(d_2(G)\le d_1(G)+r=5\)
does not exceed the admissible upper bound,
since the generator rank of \(G\) and the torsion-free Dirichlet unit rank of \(M\) with signature \((4,0)\)
are given by \(d_1(G)=\varrho_5(M)=2\) and \(r=4+0-1=3\).
\end{proof}

Examples for Case (1) are \(\mathbf{56}\) (about \(\mathbf{60}\%\)) imaginary quadratic fields \({k}_1\) starting with the discriminants \\
\(d\in\lbrace -12883,\ -13147,\ -14339,\ -23336,\ -23732,\ -26743,\ -28696,\ -35067, \\
\mbox{} \hskip 7.7cm \ -35839, \ -38984,\ -47172,\ \ldots\rbrace\).

Examples for Case (2) are \(\mathbf{23}\) (about \(\mathbf{25}\%\)) imaginary quadratic fields \(k_1\) starting with the discriminants \\
\(d\in\lbrace -27528,\ -27939,\ -39947,\ -40823,\ -54347,\ -75892,\ -91127,\ -99428, \\
\mbox{} \hskip 7.7cm \ -101784,\ -105431,\ -114679,\ \ldots\rbrace\).

Examples for Case (3) are \(\mathbf{8}\) (about \(\mathbf{9}\%\)) imaginary quadratic fields \({k}_1\) with the following discriminants \\
\(d\in\lbrace -15419,\ -16724,\ -31103,\ -42899,\ -67128,\ -70763,\ -105784,\ -194487\rbrace\).

%
%
%
%

\begin{theorem}[{\sf Two-stage tower of \(5\)-class fields with Schur \(\sigma\)-group}]
\label{thm:SchurSigmaReal}
If the \(5\)-dual field \({M}\) of \({k}_1\) has
\(5\)-principalization type \(\varkappa({M})=(124563)\), a \(4\)-cycle and two fixed points,
then the abelian type invariants of \(E_1,\ldots,E_6\) are \(\tau({M})=\lbrack (1^3)^2,(21)^4\rbrack\)
and the \(5\)-class tower group is
\(\mathrm{G}_5^{(\infty)}{{M}}=\mathrm{G}_5^{(2)}{{M}}\simeq\langle 5^5,\mathbf{11}\rangle\).
In this case, the length of the \(5\)-class tower of \({M}\) is given by \(\ell_5(M)=2\),
and \(G:=\mathrm{G}_5^{(2)}{{M}}\) is a Schur \(\sigma\)-group with balanced presentation,
that is, relation rank \(d_2(G)=d_1(G)=\varrho_5({M})=2\). \\
\end{theorem}

\begin{proof}
Similar to the proof of Theorem
\ref{thm:SchurSigmaImag}.
\end{proof}

The \textbf{unique} example is the imaginary quadratic field \({k}_1\) with discriminant
\(d=-114303\).

%
%
%
%

\begin{theorem}[{\sf Single-stage towers of \(5\)-class fields with abelian group}]
\label{thm:AbelianReal}
For the \(\mathbf{5}\) (about \(\mathbf{5}\%\)) imaginary quadratic fields \({k}_1\) with discriminants
$$d\in\lbrace -58424,\ -115912,\ -148507,\\ \ -151879,\ -154408\rbrace,$$
the \(5\)-dual field \({M}\) of \({k}_1\) has
\(5\)-principalization type \(\varkappa({M})=(000000)\), a constant with six total capitulations,
the abelian type invariants of \(E_1,\ldots,E_6\) are  \(\tau({M})=\lbrack (1)^6\rbrack\),
and the \(5\)-class tower is abelian with group
\(\mathrm{G}_5^{(\infty)}{{M}}=\mathrm{G}_5^{(1)}{{M}}\simeq\langle 5^2,\mathbf{2}\rangle\) and length \(\ell_5(M)=1\).
\end{theorem}

\begin{proof}
Similar to the proof of Theorem
\ref{thm:AbelianImag}.
\end{proof}

%
%
%
%

\begin{theorem}[{\sf Frobenius and non-Galois extensions}]
\label{thm:FrobeniusReal}
The properties of the absolute extensions \({E}_i/\mathbb{Q}\)
and the values of the invariants in the Quintic Reflection Theorem, Table
\ref{tbl:5RankCasesReal},
and Proposition
\ref{prp:FrobeniusFieldsReal},
for the \(\mathbf{93}\) cases in Proposition
\ref{prp:CountReal}
are the following ones:

$(1)$
For the \(\mathbf{5}\) cases with \(\ell_5(M)=1\) in Theorem
\ref{thm:AbelianReal},
we have
\[
(r_1,r_2,\delta_1,\delta_2)=(2,0,0,0),
\text{ and }\mathrm{Gal}({E}_i/\mathbb{Q})\simeq F_{5,2}\text{ for }1\le i\le 6\;(\text{Case } (\mathrm{a})).
\]

($2)
$For the other \(\mathbf{88}\) cases,
including the \(\mathbf{87}\) cases of \(\ell_5(M)=2\) in Theorem
\ref{thm:LowSigmaReal},
and the \textbf{unique} case of \(\ell_5(M)=2\) in Theorem
\ref{thm:SchurSigmaReal},
we have
pairwise conjugate non-Galois extensions 
$${E}_3\simeq {E}_4, {E}_5\simeq {E}_6,
\mathrm{Gal}({E}_1/ \mathbb{Q})\simeq F_{5,2},\ \mathrm{Gal}({E}_2/ \mathbb{Q})\simeq F_{5,3},$$
  and 
$$
\begin{cases}
(r_1,r_2,\delta_1,\delta_2)=(2,1,1,0),\text{ for }  d\in\lbrace -39947,\ -64103,\ -67128,\ -104503,\ -119191\rbrace\;(\text{Case } (\mathrm{d})),\\
(r_1,r_2,\delta_1,\delta_2)=(1,2,0,1),\text{ for }  d\in\lbrace -110479,\ -199735\rbrace\;(\text{Case } (\mathrm{e})),\\
(r_1,r_2,\delta_1,\delta_2)=(1,1,0,0),\text{ otherwise}\;(\text{Case } (\mathrm{c})).
\end{cases}
$$
 
\end{theorem}

\begin{proof}
See Tables
\ref{tbl:CycQrtFld3},
\ref{tbl:CycQrtFld4}
and
\ref{tbl:CycQrtFld5}.
\end{proof}

%
%
%
%
%
%
%
%
%
%
%
%
%
%
%

\bigskip
\noindent
Figure
\ref{fig:5MirrMinDisc}
visualizes the relevant part of the descendant tree of finite \(5\)-groups,
beginning at the abelian root \(C_5\times C_5=\langle 5^2,2\rangle\),
on which the second \(5\)-class groups \(\mathrm{G}_5^{(2)}{M}\)
of the fields \mbox{\(M=\mathbb{Q}\left((\zeta_5-\zeta_5^{-1})\sqrt{d}\right)\)}
are located as vertices.
The figure is a modification of the diagram in
\cite[Fig.\,3.8, p.\,448]{Ma4}.
The minimal positive, resp. maximal negative, discriminants \(d\)
are indicated by underlined boldface integers
adjacent to the oval surrounding the vertex realized by \(\mathrm{G}_5^{(2)}{M}\).
The identifiers are due to the packages
\cite{BEO2,GNO}
which are implemented in
\cite{MAGMA}.
(For trees, see
\cite{Ma6}.)

\begin{figure}[ht]
\caption{{\sf Tree position of second \(5\)-class groups \(\mathrm{G}_5^{(2)}{M}\) of the fields \(M=\mathbb{Q}\left((\zeta_5-\zeta_5^{-1})\sqrt{d}\right)\)}}
\label{fig:5MirrMinDisc}


\setlength{\unitlength}{0.9cm}
\begin{picture}(16,17)(-1,-14)

{\small


\put(0,2.5){\makebox(0,0)[cb]{Order \(5^n\)}}
\put(0,2){\line(0,-1){14}}
\multiput(-0.1,2)(0,-2){8}{\line(1,0){0.2}}
\put(-0.2,2){\makebox(0,0)[rc]{\(25\)}}
\put(0.2,2){\makebox(0,0)[lc]{\(5^2\)}}
\put(-0.2,0){\makebox(0,0)[rc]{\(125\)}}
\put(0.2,0){\makebox(0,0)[lc]{\(5^3\)}}
\put(-0.2,-2){\makebox(0,0)[rc]{\(625\)}}
\put(0.2,-2){\makebox(0,0)[lc]{\(5^4\)}}
\put(-0.2,-4){\makebox(0,0)[rc]{\(3\,125\)}}
\put(0.2,-4){\makebox(0,0)[lc]{\(5^5\)}}
\put(-0.2,-6){\makebox(0,0)[rc]{\(15\,625\)}}
\put(0.2,-6){\makebox(0,0)[lc]{\(5^6\)}}
\put(-0.2,-8){\makebox(0,0)[rc]{\(78\,125\)}}
\put(0.2,-8){\makebox(0,0)[lc]{\(5^7\)}}
\put(-0.2,-10){\makebox(0,0)[rc]{\(390\,625\)}}
\put(0.2,-10){\makebox(0,0)[lc]{\(5^8\)}}
\put(-0.2,-12){\makebox(0,0)[rc]{\(1\,953\,125\)}}
\put(0.2,-12){\makebox(0,0)[lc]{\(5^9\)}}
\put(0,-12){\vector(0,-1){2}}


\put(2.8,2.2){\makebox(0,0)[rc]{\(C_5\times C_5\)}}
\put(2.9,1.9){\framebox(0.2,0.2){}}
\put(3.2,2.2){\makebox(0,0)[lc]{\(\langle 2\rangle\)}}

\put(3,2){\line(0,-1){2}}
\put(3,0){\circle{0.2}}
\put(3.2,0.2){\makebox(0,0)[lc]{\(\langle 3\rangle\)}}


\put(3,0){\line(0,-1){2}}
\put(3,-2){\circle{0.2}}
\put(3.2,-1.8){\makebox(0,0)[lc]{\(\langle 7\rangle\)}}

\put(3,0){\line(-1,-2){1}}
\put(2,-2){\circle{0.2}}
\put(1.8,-1.8){\makebox(0,0)[rc]{\(\langle 8\rangle\)}}

\put(3,-2){\vector(0,-1){2}}
\put(3,-4.1){\makebox(0,0)[ct]{\(\mathcal{T}^1(\langle 5^2,2\rangle)\)}}


\put(3,0){\line(1,-2){2}}
\put(3,0){\line(3,-4){3}}
\put(3,0){\line(1,-1){4}}

\put(3,0){\line(2,-1){8}}
\put(3,0){\line(5,-2){10}}

\put(8,-0.5){\makebox(0,0)[lc]{Edges of depth \(2\) forming}}
\put(8,-1){\makebox(0,0)[lc]{the interface between}}
\put(8,-1.5){\makebox(0,0)[lc]{coclass \(1\) and coclass \(2\)}}


\put(8.5,-3.1){\makebox(0,0)[cc]{Hall's isoclinism family}}
\put(8.5,-3.6){\makebox(0,0)[cc]{\(\Phi_6\)}}

\multiput(5,-4)(1,0){3}{\circle*{0.2}}
\put(5,-4.1){\makebox(0,0)[rt]{\(\langle 14\rangle\)}}
\put(6,-4.1){\makebox(0,0)[rt]{\(\langle 11\rangle\)}}
\put(7,-4.1){\makebox(0,0)[rt]{\(\langle 7\rangle\)}}

\multiput(11,-4)(2,0){2}{\circle*{0.2}}
\put(11,-4.1){\makebox(0,0)[rt]{\(\langle 4\rangle\)}}
\put(13,-4.1){\makebox(0,0)[rt]{\(\langle 3\rangle\)}}


\put(7,-4){\line(0,-1){2}}
\put(7,-6){\circle*{0.1}}
\put(7,-6.1){\makebox(0,0)[rt]{\(6*\)}}
\put(6.9,-5.9){\makebox(0,0)[rc]{\(\langle 674\rangle\)}}

\multiput(11,-4)(2,0){2}{\line(0,-1){2}}
\multiput(11,-6)(2,0){2}{\circle*{0.2}}

\put(13,-4){\line(1,-2){1}}
\put(14,-6){\circle*{0.1}}
\put(14,-6.1){\makebox(0,0)[lt]{\(*2\)}}



\put(10.9,-5.9){\makebox(0,0)[rc]{\(\langle 564\rangle\)}}
\put(11,-8){\vector(0,-1){2}}
\put(11,-10.1){\makebox(0,0)[ct]{\(\mathcal{T}^2(\langle 5^5,4\rangle)\)}}

\put(11,-6){\line(-1,-2){1}}
\put(9.9,-8.3){\makebox(0,0)[cc]{\(\langle 885\rangle\)}}
\put(11,-6){\line(1,-2){1}}
\put(12.1,-7.0){\makebox(0,0)[cc]{\(\langle 891\rangle\)}}
\put(12.1,-7.3){\makebox(0,0)[cc]{\(\langle 894\rangle\)}}
\put(12.1,-7.6){\makebox(0,0)[cc]{\(\langle 897\rangle\)}}
\put(12,-8.1){\makebox(0,0)[lt]{\(*3\)}}

\multiput(11,-6)(2,0){2}{\line(0,-1){2}}
\multiput(10,-8)(1,0){4}{\circle*{0.2}}


\put(12.9,-5.9){\makebox(0,0)[rc]{\(\langle 555\rangle\)}}
\put(14.1,-5.4){\makebox(0,0)[lc]{\(\langle 669\rangle\)}}
\put(14.1,-5.9){\makebox(0,0)[lc]{\(\langle 672\rangle\)}}
\put(13,-8){\vector(0,-1){2}}
\put(13,-10.1){\makebox(0,0)[ct]{\(\mathcal{T}^2(\langle 5^5,3\rangle)\)}}

\put(13,-4){\line(1,-4){1}}
\put(14,-8){\circle*{0.2}}
\put(14.1,-7.9){\makebox(0,0)[lc]{\(\langle 115\rangle\)}}

\put(14,-8){\line(1,-4){1}}
\put(15,-12){\circle*{0.2}}
\put(15.1,-11.0){\makebox(0,0)[cc]{\(\#2;1345\)}}
\put(15.1,-11.3){\makebox(0,0)[cc]{\(\#2;1384\)}}
\put(15.1,-11.6){\makebox(0,0)[cc]{\(\#2;1420\)}}
\put(15,-12.1){\makebox(0,0)[lt]{\(*3\)}}


\put(4,-7){\makebox(0,0)[cc]{\(\varkappa=\)}}
\put(5,-7.5){\makebox(0,0)[rc]{\((123456)\)}}
\put(5,-8){\makebox(0,0)[rc]{identity}}
\put(6,-7.5){\makebox(0,0)[cc]{\((124563)\)}}
\put(6,-8){\makebox(0,0)[cc]{\(4\)-cycle}}
\put(7,-7.5){\makebox(0,0)[lc]{\((124365)\)}}
\put(7,-8){\makebox(0,0)[lc]{two \(2\)-cycles}}

\put(3.3,-8.2){\framebox(6,1.5){}}


\put(1.8,-5.25){\makebox(0,0)[cc]{\(\varkappa=\)}}
\put(1.8,-5.75){\makebox(0,0)[cc]{\((100000)\)}}
\put(1.8,-6.25){\makebox(0,0)[cc]{a.2}}
\put(3.2,-5.75){\makebox(0,0)[cc]{\((000000)\)}}
\put(3.2,-6.25){\makebox(0,0)[cc]{a.1}}
\put(0.8,-6.6){\framebox(3.4,1.5){}}

\put(11,-13.25){\makebox(0,0)[cc]{\(\varkappa=\)}}
\put(11,-13.75){\makebox(0,0)[cc]{\((022222)\)}}
\put(11,-14.25){\makebox(0,0)[cc]{nearly const.}}
\put(13,-13.75){\makebox(0,0)[cc]{\((000000)\)}}
\put(13,-14.25){\makebox(0,0)[cc]{constant}}
\put(15,-13.75){\makebox(0,0)[cc]{\((002001)\)}}
\put(15,-14.25){\makebox(0,0)[cc]{coclass \(4\)}}
\put(9.8,-14.5){\framebox(6,1.5){}}


\put(4,2.2){\makebox(0,0)[lc]{\(\underbar{\textbf{+4357}}\)}}
\put(3.3,2.1){\oval(1,0.8)}
\put(4,1.7){\makebox(0,0)[lc]{\(\underbar{\textbf{-58424}}\)}}

\put(3.3,0.1){\oval(1,0.8)}
\put(4,0){\makebox(0,0)[lc]{\(\underbar{\textbf{-12883}}\)}}

\put(1.8,-1.9){\oval(1,0.8)}
\put(2,-2.7){\makebox(0,0)[rc]{\(\underbar{\textbf{-27528}}\)}}
\put(3.3,-1.9){\oval(1,0.8)}
\put(3,-2.7){\makebox(0,0)[cc]{\(\underbar{\textbf{-15419}}\)}}

\put(4.7,-4.2){\oval(1,0.8)}
\put(4.6,-4.9){\makebox(0,0)[cc]{\(\underbar{\textbf{+581}}\)}}
\put(5.7,-4.2){\oval(1,0.8)}
\put(5.9,-4.9){\makebox(0,0)[cc]{\(\underbar{\textbf{+457}}\)}}
\put(5.9,-5.4){\makebox(0,0)[cc]{\(\underbar{\textbf{-114303}}\)}}

\put(6.9,-4.2){\oval(1,0.8)}
\put(7.1,-4.9){\makebox(0,0)[lc]{\(\underbar{\textbf{+508}}\)}}

\put(10.9,-4.2){\oval(1,0.8)}
\put(11.1,-4.9){\makebox(0,0)[lc]{\(\underbar{\textbf{+257}}\)}}
\put(9.9,-8.2){\oval(1,0.8)}
\put(9.9,-8.9){\makebox(0,0)[cc]{\(\underbar{\textbf{+4709}}\)}}
\put(12.1,-7.5){\oval(1,1.8)}
\put(12.1,-8.9){\makebox(0,0)[cc]{\(\underbar{\textbf{+1137}}\)}}

\put(14.4,-5.7){\oval(1.2,1.5)}
\put(14.4,-6.7){\makebox(0,0)[cc]{\(\underbar{\textbf{+4697}}\)}}

\put(15.1,-11.5){\oval(1.7,2.0)}
\put(15.1,-12.7){\makebox(0,0)[cc]{\(\underbar{\textbf{+8689}}\)}}

}

\end{picture}
\end{figure}

%
%
%
%
%
%
%
%
\mbox{} \bigskip
\section{Tables of second \(5\)-class groups \(\mathrm{G}_5^{(2)}{{M}}\)}
\label{s:Tables}

\subsection{Imaginary cyclic quartic fields \(M\) with \(d>0\)}
\label{ss:ImaginaryExamples}

\noindent
Table
\ref{tbl:CycQrtFld1},
resp. Table
\ref{tbl:CycQrtFld2},
shows
the factorized fundamental discriminant \(d\)
of the dual quadratic field \({k}_1\),
the \(5\)-principalization type \(\varkappa=\varkappa({M})\),
the second \(5\)-class group \(\mathrm{G}_5^{(2)}{{M}}\),
the length \(\ell_5{M}\) of the \(5\)-class tower,
the \(5\)-class ranks \(r_1:=\varrho_5({k}_1)\), \(r_2:=\varrho_5({k}_2)\),
the invariants \(\delta_1\), \(\delta_2\),
and the case in Proposition
\ref{prp:FrobeniusFieldsImag}
for the \(37\), resp. \(46\),
cyclic quartic fields \({M}=\mathbb{Q}\left((\zeta_5-\zeta_5^{-1})\sqrt{d}\right)\)
with \(0<d<5000\),
resp. \(5000<d<10000\).

%
%
%
%

For the fields with constant \(5\)-principalization type, consisting of partial kernels,
we have a \textit{polarization} of the target type whose abelian invariants can be either homogeneous \((1^5)\) or inhomogeneous \((21^3)\).
In the inhomogeneous case, there are three possibilities for the second \(5\)-class group,
namely \(\langle 5^7,891\rangle\), \(\langle 5^7,894\rangle\) and \(\langle 5^7,897\rangle\).
In the homogeneous case, the second \(5\)-class group \(\langle 5^7,885\rangle\) is unique.
According to the Shafarevich Theorem
\cite[Th. 6, Eqn. (18\({}^\prime\))]{Sh},
whose misprint we have corrected in
\cite[Th.\,5.1, p.\,28]{Ma10},
these four groups, which possess relation rank \(d_2=4\),
are forbidden as \(5\)-class tower groups for imaginary cyclic quartic fields with unit rank \(1\).
Therefore, the length of the \(5\)-class tower must be \(\ell_5{M}\ge 3\) at least,
and we conjecture a precise three-stage tower \(\ell_5{M}=3\).

The \textit{complete statistics} of the \(83\) \textit{imaginary} cyclic quartic fields \({M}\) with \(0<d<10^4\) is as follows:\\

\begin{tabular}{lll}
{\bf -- }There are \(23\) (about \(28\%)\) cases with \(\mathrm{G}_5^{(2)}{{M}}\simeq\langle 5^5,11\rangle\), the Schur \(\sigma\)-group with transfer \\
$\quad$kernel type a \(4\)-cycle.\\

 {\bf -- }There are 
\(22\) (about \(27\%)\) cases with \(\mathrm{G}_5^{(2)}{{M}}\simeq\langle 5^5,4\rangle\).\\

{\bf -- }There are 
\(16\) (about \(19\%)\) cases with \(\mathrm{G}_5^{(2)}{{M}}\simeq\langle 5^5,7\rangle\).\\

{\bf -- }There are 
\(11\) (about \(13\%)\) cases with \(\mathrm{G}_5^{(2)}{{M}}\simeq\langle 5^5,14\rangle\),  the Schur \(\sigma\)-group with transfer\\
 $\quad$kernel type the identity permutation.\\

{\bf -- }For only \(4\) cases we have \(\mathrm{G}_5^{(2)}{{M}}\simeq\langle 5^7,885\rangle\).\\

{\bf -- }For \(2\) cases \(\mathrm{G}_5^{(2)}{{M}}\simeq\langle 5^7,891\vert 894\vert 897\rangle\).\\

{\bf -- }For \(2\) cases \(\mathrm{G}_5^{(2)}{{M}}\simeq\langle 5^2,2\rangle\), the elementary bicyclic \(5\)-group of rank \(2\).\\

{\bf -- }For \(2\) cases \(\mathrm{G}_5^{(2)}{{M}}\) is a descendant of \(\langle 5^5,3\rangle\) indicated by the symbol \(\downarrow\).\\

{\bf -- }The last three groups have the biggest order \(5^9\) and coclass \(4\).\\
$\quad$They posses relation rank \(d_2=5\), which clearly enforces \(\ell_5{M}\ge 3\) by the Shafarevich \\
$\quad$Theorem. Again, we conjecture the equality \(\ell_5{M}=3\).\\
\end{tabular}\bigskip \bigskip

Furthermore,
we point out that the groups \(\langle 5^5,4\rangle\) and \(\langle 5^5,7\rangle\) with relation rank \(d_2=3\)
are not strong \(\sigma\)-groups in the sense of Schoof
\cite{Sf}.
They are forbidden as \(5\)-class tower groups for any quadratic field, imaginary or real.
However, they are admissible for our imaginary cyclic quartic fields \({M}\) with unit rank \(1\),
since the subfield \(k_0^+=\mathbb{Q}(\sqrt{5})\) also possesses unit rank \(1\),
and so a strong \(\sigma\)-group is not required.

%
%
%
%
%
%
%
%
\mbox{} 
\subsection{Real cyclic quartic fields \(M\) with \(d<0\)}
\label{ss:RealExamples}

\noindent
Table
\ref{tbl:CycQrtFld3},
resp. Table
\ref{tbl:CycQrtFld4},
resp. Table
\ref{tbl:CycQrtFld5},
shows
the factorized fundamental discriminant \(d\)
of the dual quadratic field \({k}_1\),
the \(5\)-principalization type \(\varkappa=\varkappa({M})\),
the \(5\)-class tower group \(\mathrm{G}_5^{(\infty)}{{M}}\),
the length \(\ell_5{M}\) of the \(5\)-class tower,
the \(5\)-class ranks \(r_1:=\varrho_5({k}_1)\), \(r_2:=\varrho_5({k}_2)\),
the invariants \(\delta_1\), \(\delta_2\),
and the case in Proposition
\ref{prp:FrobeniusFieldsReal}
for the \(38\), resp. \(38\), resp. \(17\),
cyclic quartic fields \({M}=\mathbb{Q}\left((\zeta_5-\zeta_5^{-1})\sqrt{d}\right)\)
with \(-100000<d<0\),
resp. \(-175000<d<-100000\),
resp. \(-200000<d<-175000\).\bigskip

%
%
%
%

The \textit{complete statistics} of the \(93\) \textit{real} cyclic quartic fields \({M}\) with \(-2\cdot 10^5<d<0\) is as follows:\\

\begin{tabular}{lll}
{\bf --} There are \(56\) (about \(60\%)\) cases with \(\mathrm{G}_5^{(\infty)}{{M}}\simeq\langle 5^3,3\rangle\) the extra special \(5\)-group of \\
$\quad$exponent \(5\).\\
{\bf --} There are \(23\) (about\(25\%)\) cases with \(\mathrm{G}_5^{(\infty)}{{M}}\simeq\langle 5^4,8\rangle\) having a transfer kernel type with \\
$\quad$fixed point.\\ 
{\bf --} There are \(8\) (about \(9\%)\) cases with \(\mathrm{G}_5^{(\infty)}{{M}}\simeq\langle 5^4,7\rangle\) having total transfer kernels \\
$\quad$exclusively.\\
{\bf --} For only \(5\) cases we have \(\mathrm{G}_5^{(\infty)}{{M}}\simeq\langle 5^2,2\rangle\) the elementary bicyclic \(5\)-group of rank \(2\).\\
{\bf --} For a unique case \(\mathrm{G}_5^{(\infty)}{{M}}\simeq\langle 5^5,11\rangle\) the Schur \(\sigma\)-group with transfer kernel type a \(4\)-cycle.\\
\end{tabular}

\medskip
The \(5\)-class tower of \(M\) possesses length \(\ell_5{M}=1\) for the abelian \(\mathrm{G}_5^{(\infty)}{{M}}\simeq\langle 5^2,2\rangle\),\\
and \(\ell_5{M}=2\) in all other cases.

\bigskip
According to the Shafarevich Theorem
\cite[Thm. 6, Eqn. (18\({}^\prime\))]{Sh},
whose misprint we have corrected in
\cite[Thm.\,5.1, p.\,28]{Ma10},
the mainline groups \(\langle 5^3,3\rangle\) and \(\langle 5^4,7\rangle\) with relation rank \(d_2=4\)
are forbidden as \(5\)-class tower groups for real quadratic fields with unit rank \(1\),
but they are admissible for real cyclic quartic fields, which have bigger unit rank \(3\).

%
%
%
%
%
%
%
%

\subsection{The Galois action confirmed}
\label{ss:ConfirmedGaloisAction}
\noindent
All numerical results in the Tables
\ref{tbl:CycQrtFld1}
to
\ref{tbl:CycQrtFld5}
are in perfect accordance with Theorems
\ref{thm:CycQuart},
\ref{thm:Sigma5GroupsDegree4}
and Corollary
\ref{cor:ParentOperator}.
A rigorous check with the computational algebra system MAGMA
\cite{MAGMA,BEO2}
proves that only the two terminal Schur \(\sigma\)-groups
\(\langle 5^5,11\rangle\), \(\langle 5^5,14\rangle\)
and five other capable top vertices
\(\langle 5^5,3\rangle\), \(\langle 5^5,4\rangle\), \(\langle 5^5,5\rangle\), \(\langle 5^5,6\rangle\), \(\langle 5^5,7\rangle\)
in the stem of Hall's isoclinism family \(\Phi_6\),
and the abelian root \(\langle 5^2,2\rangle\), together with their descendants
\cite{Ma9},
are admissible for \(\mathrm{G}_5^{(2)}{{M}}\) of any cyclic quartic field \({M}\),
as drawn in Figure
\ref{fig:5MirrMinDisc}.

%
%
%
%


%
%
%
%

\renewcommand{\arraystretch}{1.0}
\begin{small}
\begin{table}[ht]
\caption{{\sf The group \(\mathrm{G}_5^{(2)}{{M}}\) of \({M}=\mathbb{Q}\left((\zeta_5-\zeta_5^{-1})\sqrt{d}\right)\)  with  \(0<d<5000\)}} 
\label{tbl:CycQrtFld1}
\begin{center}
\begin{tabular}{|r|r|c||c|l||c|c||c|c|c|c|c|}
\hline
 No.    & \multicolumn{2}{|c||}{Discriminant} & \multicolumn{2}{|c||}{Principalization} & \(\mathrm{G}_5^{(2)}{{M}}\) & \(\ell_5{M}\) & \multicolumn{5}{|c|}{Invariants} \\
        &    \(d\) & Factors      & \(\varkappa\) & Remark                  &                            &       & \(r_1\) & \(\delta_1\) & \(r_2\) & \(\delta_2\) & Case \\
\hline
  \(1\) &  \(257\) & prime        &  \((660666)\) & nearly const.           & \(\langle 5^5,4\rangle\)   & \(2\) &   \(0\) &        \(0\) &    \(0\)&        \(0\) &  (d) \\
  \(2\) &  \(457\) & prime        &  \((234156)\) & \(4\)-cycle             & \(\langle 5^5,11\rangle\)  & \(2\) &   \(0\) &        \(0\) &    \(0\)&        \(0\) &  (d) \\
  \(3\) &  \(501\) & \(3,167\)    &  \((521346)\) & \(4\)-cycle             & \(\langle 5^5,11\rangle\)  & \(2\) &   \(0\) &        \(0\) &    \(0\)&        \(0\) &  (d) \\
  \(4\) &  \(508\) & \(4,127\)    &  \((653421)\) & two \(2\)-cycles        & \(\langle 5^5,7\rangle\)   & \(2\) &   \(0\) &        \(0\) &    \(0\)&        \(0\) &  (d) \\
  \(5\) &  \(509\) & prime        &  \((216453)\) & two \(2\)-cycles        & \(\langle 5^5,7\rangle\)   & \(2\) &   \(0\) &        \(0\) &    \(0\)&        \(0\) &  (d) \\
  \(6\) &  \(581\) & \(7,83\)     &  \((123456)\) & identity                & \(\langle 5^5,14\rangle\)  & \(2\) &   \(0\) &        \(0\) &    \(0\)&        \(0\) &  (d) \\
  \(7\) &  \(629\) & \(17,37\)    &  \((154326)\) & two \(2\)-cycles        & \(\langle 5^5,7\rangle\)   & \(2\) &   \(0\) &        \(0\) &    \(0\)&        \(0\) &  (d) \\
  \(8\) &  \(753\) & \(3,251\)    &  \((123456)\) & identity                & \(\langle 5^5,14\rangle\)  & \(2\) &   \(0\) &        \(0\) &    \(0\)&        \(0\) &  (d) \\
  \(9\) &  \(764\) & \(4,191\)    &  \((666066)\) & nearly const.           & \(\langle 5^5,4\rangle\)   & \(2\) &   \(0\) &        \(0\) &    \(0\)&        \(0\) &  (d) \\
 \(10\) &  \(881\) & prime        &  \((463152)\) & two \(2\)-cycles        & \(\langle 5^5,7\rangle\)   & \(2\) &   \(0\) &        \(0\) &    \(0\)&        \(0\) &  (d) \\
 \(11\) & \(1113\) & \(3,7,53\)   &  \((653421)\) & two \(2\)-cycles        & \(\langle 5^5,7\rangle\)   & \(2\) &   \(0\) &        \(0\) &    \(0\)&        \(0\) &  (d) \\
 \(12\) & \(1137\) & \(3,379\)    &  \((444444)\) & constant                & \(\langle 5^7,891\vert 894\vert 897\rangle\) & \(\ge 3\) & \(0\) & \(0\) & \(0\) & \(0\) &  (d) \\
 \(13\) & \(1192\) & \(8,149\)    &  \((463152)\) & two \(2\)-cycles        & \(\langle 5^5,7\rangle\)   & \(2\) &   \(0\) &        \(0\) &    \(0\)&        \(0\) &  (d) \\
 \(14\) & \(1704\) & \(8,3,71\)   &  \((653421)\) & two \(2\)-cycles        & \(\langle 5^5,7\rangle\)   & \(2\) &   \(0\) &        \(0\) &    \(0\)&        \(0\) &  (d) \\
 \(15\) & \(1708\) & \(4,7,61\)   &  \((404444)\) & nearly const.           & \(\langle 5^5,4\rangle\)   & \(2\) &   \(0\) &        \(0\) &    \(0\)&        \(0\) &  (d) \\
 \(16\) & \(1829\) & \(31,59\)    &  \((216453)\) & two \(2\)-cycles        & \(\langle 5^5,7\rangle\)   & \(2\) &   \(0\) &        \(0\) &    \(0\)&        \(0\) &  (d) \\
 \(17\) & \(1853\) & \(17,109\)   &  \((550555)\) & nearly const.           & \(\langle 5^5,4\rangle\)   & \(2\) &   \(0\) &        \(0\) &    \(0\)&        \(0\) &  (d) \\
 \(18\) & \(1996\) & \(4,499\)    &  \((613254)\) & \(4\)-cycle             & \(\langle 5^5,11\rangle\)  & \(2\) &   \(1\) &        \(1\) &    \(0\)&        \(0\) &  (e) \\
 \(19\) & \(2008\) & \(8,251\)    &  \((550555)\) & nearly const.           & \(\langle 5^5,4\rangle\)   & \(2\) &   \(0\) &        \(0\) &    \(0\)&        \(0\) &  (d) \\
 \(20\) & \(2189\) & \(11,199\)   &  \((505555)\) & nearly const.           & \(\langle 5^5,4\rangle\)   & \(2\) &   \(0\) &        \(0\) &    \(0\)&        \(0\) &  (d) \\
 \(21\) & \(2296\) & \(8,7,41\)   &  \((123456)\) & identity                & \(\langle 5^5,14\rangle\)  & \(2\) &   \(0\) &        \(0\) &    \(0\)&        \(0\) &  (d) \\
 \(22\) & \(2573\) & \(31,83\)    &  \((613254)\) & \(4\)-cycle             & \(\langle 5^5,11\rangle\)  & \(2\) &   \(0\) &        \(0\) &    \(0\)&        \(0\) &  (d) \\
 \(23\) & \(2829\) & \(3,23,41\)  &  \((123456)\) & identity                & \(\langle 5^5,14\rangle\)  & \(2\) &   \(0\) &        \(0\) &    \(0\)&        \(0\) &  (d) \\
 \(24\) & \(3121\) & prime        &  \((532416)\) & two \(2\)-cycles        & \(\langle 5^5,7\rangle\)   & \(2\) &   \(1\) &        \(1\) &    \(0\)&        \(0\) &  (e) \\
 \(25\) & \(3129\) & \(3,7,149\)  &  \((333303)\) & nearly const.           & \(\langle 5^5,4\rangle\)   & \(2\) &   \(1\) &        \(1\) &    \(0\)&        \(0\) &  (e) \\
 \(26\) & \(3169\) & prime        &  \((444444)\) & constant                & \(\langle 5^7,891\vert 894\vert 897\rangle\) & \(\ge 3\) & \(0\) & \(0\) & \(0\) & \(0\) &  (d) \\
 \(27\) & \(3253\) & prime        &  \((243651)\) & \(4\)-cycle             & \(\langle 5^5,11\rangle\)  & \(2\) &   \(1\) &        \(1\) &    \(0\)&        \(0\) &  (e) \\
 \(28\) & \(4189\) & \(59,71\)    &  \((243651)\) & \(4\)-cycle             & \(\langle 5^5,11\rangle\)  & \(2\) &   \(0\) &        \(0\) &    \(0\)&        \(0\) &  (d) \\
 \(29\) & \(4357\) & prime        &  \((000000)\) & abelian                 & \(\langle 5^2,2\rangle\)   & \(1\) &   \(1\) &        \(0\) &    \(0\)&        \(1\) &  (a) \\
 \(30\) & \(4444\) & \(4,11,101\) &  \((000000)\) & abelian                 & \(\langle 5^2,2\rangle\)   & \(1\) &   \(1\) &        \(0\) &    \(0\)&        \(1\) &  (a) \\
 \(31\) & \(4461\) & \(3,1487\)   &  \((653421)\) & two \(2\)-cycles        & \(\langle 5^5,7\rangle\)   & \(2\) &   \(0\) &        \(0\) &    \(0\)&        \(0\) &  (d) \\
 \(32\) & \(4504\) & \(8,563\)    &  \((444404)\) & nearly const.           & \(\langle 5^5,4\rangle\)   & \(2\) &   \(1\) &        \(1\) &    \(1\)&        \(1\) &  (c) \\
 \(33\) & \(4553\) & \(29,157\)   &  \((123456)\) & identity                & \(\langle 5^5,14\rangle\)  & \(2\) &   \(0\) &        \(0\) &    \(0\)&        \(0\) &  (d) \\
 \(34\) & \(4697\) & \(7,11,61\)  &  \((000000)\) & tot., non-ab.           & \(\langle 5^5,3\rangle\downarrow\) & \(\ge 3\) & \(0\) &  \(0\) &    \(0\)&        \(0\) &  (d) \\
 \(35\) & \(4709\) & \(17,277\)   &  \((444444)\) & constant                & \(\langle 5^7,885\rangle\) & \(\ge 3\) &   \(0\) &        \(0\) &    \(0\)&        \(0\) &  (d) \\
 \(36\) & \(4861\) & prime        &  \((333303)\) & nearly const.           & \(\langle 5^5,4\rangle\)   & \(2\) &   \(0\) &        \(0\) &    \(0\)&        \(0\) &  (d) \\
 \(37\) & \(4957\) & prime        &  \((135246)\) & \(4\)-cycle             & \(\langle 5^5,11\rangle\)  & \(2\) &   \(0\) &        \(0\) &    \(0\)&        \(0\) &  (d) \\
\hline
\end{tabular}
\end{center}
\end{table}
\end{small}
 
%
%
%
%

\newpage

%
%

\renewcommand{\arraystretch}{1.0}
\begin{table}[ht]
\caption{{\sf The group \(\mathrm{G}_5^{(2)}{{M}}\) of \({M}=\mathbb{Q}\left((\zeta_5-\zeta_5^{-1})\sqrt{d}\right)\)  with \(5000<d<10000\)}}
  
\label{tbl:CycQrtFld2}
\begin{center}
\begin{tabular}{|r|r|c||c|l||c|c||c|c|c|c|c|}
\hline
 No.    & \multicolumn{2}{|c||}{Discriminant} & \multicolumn{2}{|c||}{Principalization} & \(\mathrm{G}_5^{(2)}{{M}}\) & \(\ell_5{M}\) & \multicolumn{5}{|c|}{Invariants} \\
        &    \(d\) & Factors      & \(\varkappa\) & Remark                  &                            &       & \(r_1\) & \(\delta_1\) & \(r_2\) & \(\delta_2\) & Case \\
\hline
 \(38\) & \(5116\) & \(4,1279\)   &  \((123456)\) & identity                & \(\langle 5^5,14\rangle\)  & \(2\) &   \(0\) &        \(0\) &    \(1\)&        \(1\) &  (f) \\
 \(39\) & \(5129\) & \(23,223\)   &  \((526431)\) & \(4\)-cycle             & \(\langle 5^5,11\rangle\)  & \(2\) &   \(0\) &        \(0\) &    \(0\)&        \(0\) &  (d) \\
 \(40\) & \(5233\) & prime        &  \((142536)\) & \(4\)-cycle             & \(\langle 5^5,11\rangle\)  & \(2\) &   \(0\) &        \(0\) &    \(0\)&        \(0\) &  (d) \\
 \(41\) & \(5241\) & \(3,1747\)   &  \((660666)\) & nearly const.           & \(\langle 5^5,4\rangle\)   & \(2\) &   \(1\) &        \(1\) &    \(0\)&        \(0\) &  (e) \\
 \(42\) & \(5269\) & \(11,479\)   &  \((222220)\) & nearly const.           & \(\langle 5^5,4\rangle\)   & \(2\) &   \(1\) &        \(1\) &    \(0\)&        \(0\) &  (e) \\
 \(43\) & \(5308\) & \(4,1327\)   &  \((513462)\) & \(4\)-cycle             & \(\langle 5^5,11\rangle\)  & \(2\) &   \(1\) &        \(1\) &    \(0\)&        \(0\) &  (e) \\
 \(44\) & \(5361\) & \(3,1787\)   &  \((625413)\) & \(4\)-cycle             & \(\langle 5^5,11\rangle\)  & \(2\) &   \(0\) &        \(0\) &    \(0\)&        \(0\) &  (d) \\
 \(45\) & \(5393\) & prime        &  \((440444)\) & nearly const.           & \(\langle 5^5,4\rangle\)   & \(2\) &   \(0\) &        \(0\) &    \(0\)&        \(0\) &  (d) \\
 \(46\) & \(5464\) & \(8,683\)    &  \((440444)\) & nearly const.           & \(\langle 5^5,4\rangle\)   & \(2\) &   \(0\) &        \(0\) &    \(0\)&        \(0\) &  (d) \\
 \(47\) & \(5557\) & prime        &  \((111111)\) & constant                & \(\langle 5^7,885\rangle\) & \(\ge 3\) &   \(0\) &        \(0\) &    \(0\)&        \(0\) &  (d) \\
 \(48\) & \(5736\) & \(8,3,239\)  &  \((123456)\) & identity                & \(\langle 5^5,14\rangle\)  & \(2\) &   \(0\) &        \(0\) &    \(0\)&        \(0\) &  (d) \\
 \(49\) & \(5989\) & \(53,113\)   &  \((440444)\) & nearly const.           & \(\langle 5^5,4\rangle\)   & \(2\) &   \(0\) &        \(0\) &    \(0\)&        \(0\) &  (d) \\
 \(50\) & \(6072\) & \(8,3,11,23\)&  \((613254)\) & \(4\)-cycle             & \(\langle 5^5,11\rangle\)  & \(2\) &   \(0\) &        \(0\) &    \(0\)&        \(0\) &  (d) \\
 \(51\) & \(6073\) & prime        &  \((000000)\) & tot., non-ab.           & \(\langle 5^5,3\rangle\downarrow\) & \(\ge 3\) & \(0\) &  \(0\) &    \(0\)&        \(0\) &  (d) \\
 \(52\) & \(6113\) & prime        &  \((421653)\) & \(4\)-cycle             & \(\langle 5^5,11\rangle\)  & \(2\) &   \(1\) &        \(1\) &    \(0\)&        \(0\) &  (e) \\
 \(53\) & \(6524\) & \(4,7,233\)  &  \((513462)\) & \(4\)-cycle             & \(\langle 5^5,11\rangle\)  & \(2\) &   \(0\) &        \(0\) &    \(0\)&        \(0\) &  (d) \\
 \(54\) & \(6761\) & prime        &  \((123456)\) & identity                & \(\langle 5^5,14\rangle\)  & \(2\) &   \(0\) &        \(0\) &    \(0\)&        \(0\) &  (d) \\
 \(55\) & \(6949\) & prime        &  \((666066)\) & nearly const.           & \(\langle 5^5,4\rangle\)   & \(2\) &   \(1\) &        \(1\) &    \(1\)&        \(1\) &  (c) \\
 \(56\) & \(6952\) & \(8,11,79\)  &  \((220222)\) & nearly const.           & \(\langle 5^5,4\rangle\)   & \(2\) &   \(0\) &        \(0\) &    \(0\)&        \(0\) &  (d) \\
 \(57\) & \(7032\) & \(8,3,293\)  &  \((213546)\) & two \(2\)-cycles        & \(\langle 5^5,7\rangle\)   & \(2\) &   \(0\) &        \(0\) &    \(0\)&        \(0\) &  (d) \\
 \(58\) & \(7041\) & \(3,2347\)   &  \((666666)\) & constant                & \(\langle 5^7,885\rangle\) & \(\ge 3\) &   \(0\) &        \(0\) &    \(0\)&        \(0\) &  (d) \\
 \(59\) & \(7221\) & \(3,29,83\)  &  \((444404)\) & nearly const.           & \(\langle 5^5,4\rangle\)   & \(2\) &   \(1\) &        \(1\) &    \(1\)&        \(1\) &  (c) \\
 \(60\) & \(7229\) & prime        &  \((444444)\) & constant                & \(\langle 5^7,885\rangle\) & \(\ge 3\) &   \(1\) &        \(1\) &    \(1\)&        \(1\) &  (c) \\
 \(61\) & \(7336\) & \(8,7,131\)  &  \((606666)\) & nearly const.           & \(\langle 5^5,4\rangle\)   & \(2\) &   \(0\) &        \(0\) &    \(0\)&        \(0\) &  (d) \\
 \(62\) & \(7361\) & \(17,433\)   &  \((653421)\) & two \(2\)-cycles        & \(\langle 5^5,7\rangle\)   & \(2\) &   \(0\) &        \(0\) &    \(0\)&        \(0\) &  (d) \\
 \(63\) & \(7489\) & prime        &  \((123456)\) & identity                & \(\langle 5^5,14\rangle\)  & \(2\) &   \(0\) &        \(0\) &    \(0\)&        \(0\) &  (d) \\
 \(64\) & \(7628\) & \(4,1907\)   &  \((164253)\) & \(4\)-cycle             & \(\langle 5^5,11\rangle\)  & \(2\) &   \(0\) &        \(0\) &    \(0\)&        \(0\) &  (d) \\
 \(65\) & \(7656\) & \(8,3,11,29\)&  \((444404)\) & nearly const.           & \(\langle 5^5,4\rangle\)   & \(2\) &   \(0\) &        \(0\) &    \(0\)&        \(0\) &  (d) \\
 \(66\) & \(7752\) & \(8,3,17,19\)&  \((623145)\) & \(4\)-cycle             & \(\langle 5^5,11\rangle\)  & \(2\) &   \(0\) &        \(0\) &    \(0\)&        \(0\) &  (d) \\
 \(67\) & \(7833\) & \(3,7,373\)  &  \((326154)\) & \(4\)-cycle             & \(\langle 5^5,11\rangle\)  & \(2\) &   \(0\) &        \(0\) &    \(0\)&        \(0\) &  (d) \\
 \(68\) & \(7996\) & \(4,1999\)   &  \((022222)\) & nearly const.           & \(\langle 5^5,4\rangle\)   & \(2\) &   \(0\) &        \(0\) &    \(0\)&        \(0\) &  (d) \\
 \(69\) & \(8008\) & \(8,7,11,13\)&  \((625413)\) & \(4\)-cycle             & \(\langle 5^5,11\rangle\)  & \(2\) &   \(0\) &        \(0\) &    \(0\)&        \(0\) &  (d) \\
 \(70\) & \(8012\) & \(4,2003\)   &  \((165432)\) & two \(2\)-cycles        & \(\langle 5^5,7\rangle\)   & \(2\) &   \(0\) &        \(0\) &    \(0\)&        \(0\) &  (d) \\
 \(71\) & \(8309\) & \(7,1187\)   &  \((111110)\) & nearly const.           & \(\langle 5^5,4\rangle\)   & \(2\) &   \(1\) &        \(1\) &    \(0\)&        \(0\) &  (e) \\
 \(72\) & \(8689\) & prime        &  \((002001)\) & coclass \(4\)           & \(\langle 5^7,115\rangle\downarrow\) & \(\ge 3\) & \(1\) & \(1\) &   \(0\)&        \(0\) &  (e) \\
 \(73\) & \(8789\) & \(11,17,47\) &  \((362451)\) & \(4\)-cycle             & \(\langle 5^5,11\rangle\)  & \(2\) &   \(0\) &        \(0\) &    \(0\)&        \(0\) &  (d) \\
 \(74\) & \(8877\) & \(3,11,269\) &  \((463152)\) & two \(2\)-cycles        & \(\langle 5^5,7\rangle\)   & \(2\) &   \(0\) &        \(0\) &    \(0\)&        \(0\) &  (d) \\
 \(75\) & \(8972\) & \(4,2243\)   &  \((362451)\) & \(4\)-cycle             & \(\langle 5^5,11\rangle\)  & \(2\) &   \(0\) &        \(0\) &    \(1\)&        \(1\) &  (f) \\
 \(76\) & \(9013\) & prime        &  \((123456)\) & identity                & \(\langle 5^5,14\rangle\)  & \(2\) &   \(0\) &        \(0\) &    \(1\)&        \(1\) &  (f) \\
 \(77\) & \(9052\) & \(4,31,73\)  &  \((333303)\) & nearly const.           & \(\langle 5^5,4\rangle\)   & \(2\) &   \(0\) &        \(0\) &    \(0\)&        \(0\) &  (d) \\
 \(78\) & \(9544\) & \(8,1193\)   &  \((125364)\) & \(4\)-cycle             & \(\langle 5^5,11\rangle\)  & \(2\) &   \(0\) &        \(0\) &    \(0\)&        \(0\) &  (d) \\
 \(79\) & \(9564\) & \(4,3,797\)  &  \((425136)\) & two \(2\)-cycles        & \(\langle 5^5,7\rangle\)   & \(2\) &   \(0\) &        \(0\) &    \(0\)&        \(0\) &  (d) \\
 \(80\) & \(9573\) & \(3,3191\)   &  \((216453)\) & two \(2\)-cycles        & \(\langle 5^5,7\rangle\)   & \(2\) &   \(0\) &        \(0\) &    \(0\)&        \(0\) &  (d) \\
 \(81\) & \(9669\) & \(3,11,293\) &  \((362451)\) & \(4\)-cycle             & \(\langle 5^5,11\rangle\)  & \(2\) &   \(1\) &        \(1\) &    \(1\)&        \(1\) &  (c) \\
 \(82\) & \(9752\) & \(8,23,53\)  &  \((513462)\) & \(4\)-cycle             & \(\langle 5^5,11\rangle\)  & \(2\) &   \(0\) &        \(0\) &    \(0\)&        \(0\) &  (d) \\
 \(83\) & \(9829\) & prime        &  \((123456)\) & identity                & \(\langle 5^5,14\rangle\)  & \(2\) &   \(1\) &        \(1\) &    \(0\)&        \(0\) &  (e) \\
\hline
\end{tabular}
\end{center}
\end{table}

%
%
%
%

\newpage

%
%
%
%

\renewcommand{\arraystretch}{1.0}

\begin{table}[ht]
\caption{{\sf The group \(\mathrm{G}_5^{(\infty)}{{M}}\) of \({M}=\mathbb{Q}\left((\zeta_5-\zeta_5^{-1})\sqrt{d}\right)\)  with \(-100000<d<0\)}}
\label{tbl:CycQrtFld3}
\begin{center}
\begin{tabular}{|r|r|c||c|l||c|c||c|c|c|c|c|}
\hline
 No.    & \multicolumn{2}{|c||}{Discriminant} & \multicolumn{2}{|c||}{Principalization} & \(\mathrm{G}_5^{(\infty)}{{M}}\) & \(\ell_5{M}\) & \multicolumn{5}{|c|}{Invariants} \\
        &       \(d\) & Factors       & \(\varkappa\) & Type                   &                            &       & \(r_1\) & \(\delta_1\) & \(r_2\) & \(\delta_2\) & Case \\
\hline
  \(1\) &  \(-12883\) &    \(13,991\) &  \((000000)\) & a.1                    & \(\langle 5^3,3\rangle\)   & \(2\) &   \(1\) &        \(0\) &    \(1\)&        \(0\) &  (c) \\
  \(2\) &  \(-13147\) &         prime &  \((000000)\) & a.1                    & \(\langle 5^3,3\rangle\)   & \(2\) &   \(1\) &        \(0\) &    \(1\)&        \(0\) &  (c) \\
  \(3\) &  \(-14339\) &   \(13,1103\) &  \((000000)\) & a.1                    & \(\langle 5^3,3\rangle\)   & \(2\) &   \(1\) &        \(0\) &    \(1\)&        \(0\) &  (c) \\
  \(4\) &  \(-15419\) &    \(17,907\) &  \((000000)\) & a.1 \(\uparrow\)       & \(\langle 5^4,7\rangle\)   & \(2\) &   \(1\) &        \(0\) &    \(1\)&        \(0\) &  (c) \\
  \(5\) &  \(-16724\) &  \(4,37,113\) &  \((000000)\) & a.1 \(\uparrow\)       & \(\langle 5^4,7\rangle\)   & \(2\) &   \(1\) &        \(0\) &    \(1\)&        \(0\) &  (c) \\
  \(6\) &  \(-23336\) &    \(8,2917\) &  \((000000)\) & a.1                    & \(\langle 5^3,3\rangle\)   & \(2\) &   \(1\) &        \(0\) &    \(1\)&        \(0\) &  (c) \\
  \(7\) &  \(-23732\) &  \(4,17,349\) &  \((000000)\) & a.1                    & \(\langle 5^3,3\rangle\)   & \(2\) &   \(1\) &        \(0\) &    \(1\)&        \(0\) &  (c) \\
  \(8\) &  \(-26743\) &    \(47,569\) &  \((000000)\) & a.1                    & \(\langle 5^3,3\rangle\)   & \(2\) &   \(1\) &        \(0\) &    \(1\)&        \(0\) &  (c) \\
  \(9\) &  \(-27528\) & \(8,3,31,37\) &  \((003000)\) & a.2, fixed point       & \(\langle 5^4,8\rangle\)   & \(2\) &   \(1\) &        \(0\) &    \(1\)&        \(0\) &  (c) \\
 \(10\) &  \(-27939\) &  \(3,67,139\) &  \((000050)\) & a.2, fixed point       & \(\langle 5^4,8\rangle\)   & \(2\) &   \(1\) &        \(0\) &    \(1\)&        \(0\) &  (c) \\
 \(11\) &  \(-28696\) &  \(8,17,211\) &  \((000000)\) & a.1                    & \(\langle 5^3,3\rangle\)   & \(2\) &   \(1\) &        \(0\) &    \(1\)&        \(0\) &  (c) \\
 \(12\) &  \(-31103\) &   \(19,1637\) &  \((000000)\) & a.1 \(\uparrow\)       & \(\langle 5^4,7\rangle\)   & \(2\) &   \(1\) &        \(0\) &    \(1\)&        \(0\) &  (c) \\
 \(13\) &  \(-35067\) &   \(3,11689\) &  \((000000)\) & a.1                    & \(\langle 5^3,3\rangle\)   & \(2\) &   \(1\) &        \(0\) &    \(1\)&        \(0\) &  (c) \\
 \(14\) &  \(-35839\) &         prime &  \((000000)\) & a.1                    & \(\langle 5^3,3\rangle\)   & \(2\) &   \(1\) &        \(0\) &    \(1\)&        \(0\) &  (c) \\
 \(15\) &  \(-38984\) &  \(8,11,443\) &  \((000000)\) & a.1                    & \(\langle 5^3,3\rangle\)   & \(2\) &   \(1\) &        \(0\) &    \(1\)&        \(0\) &  (c) \\
 \(16\) &  \(-39947\) &    \(43,929\) &  \((003000)\) & a.2, fixed point       & \(\langle 5^4,8\rangle\)   & \(2\) &   \(2\) &        \(1\) &    \(1\)&        \(0\) &  (d) \\
 \(17\) &  \(-40823\) &         prime &  \((000050)\) & a.2, fixed point       & \(\langle 5^4,8\rangle\)   & \(2\) &   \(1\) &        \(0\) &    \(1\)&        \(0\) &  (c) \\
 \(18\) &  \(-42899\) &         prime &  \((000000)\) & a.1 \(\uparrow\)       & \(\langle 5^4,7\rangle\)   & \(2\) &   \(1\) &        \(0\) &    \(1\)&        \(0\) &  (c) \\
 \(19\) &  \(-47172\) &  \(4,3,3931\) &  \((000000)\) & a.1                    & \(\langle 5^3,3\rangle\)   & \(2\) &   \(1\) &        \(0\) &    \(1\)&        \(0\) &  (c) \\
 \(20\) &  \(-52276\) &  \(4,7,1867\) &  \((000000)\) & a.1                    & \(\langle 5^3,3\rangle\)   & \(2\) &   \(1\) &        \(0\) &    \(1\)&        \(0\) &  (c) \\
 \(21\) &  \(-54347\) &         prime &  \((100000)\) & a.2, fixed point       & \(\langle 5^4,8\rangle\)   & \(2\) &   \(1\) &        \(0\) &    \(1\)&        \(0\) &  (c) \\
 \(22\) &  \(-55667\) &         prime &  \((000000)\) & a.1                    & \(\langle 5^3,3\rangle\)   & \(2\) &   \(1\) &        \(0\) &    \(1\)&        \(0\) &  (c) \\
 \(23\) &  \(-56167\) &         prime &  \((000000)\) & a.1                    & \(\langle 5^3,3\rangle\)   & \(2\) &   \(1\) &        \(0\) &    \(1\)&        \(0\) &  (c) \\
 \(24\) &  \(-58424\) &  \(8,67,109\) &  \((000000)\) & abelian                & \(\langle 5^2,2\rangle\)   & \(1\) &   \(2\) &        \(0\) &    \(0\)&        \(0\) &  (a) \\
 \(25\) &  \(-64103\) &   \(13,4931\) &  \((000000)\) & a.1                    & \(\langle 5^3,3\rangle\)   & \(2\) &   \(2\) &        \(1\) &    \(1\)&        \(0\) &  (d) \\
 \(26\) &  \(-64724\) & \(4,11,1471\) &  \((000000)\) & a.1                    & \(\langle 5^3,3\rangle\)   & \(2\) &   \(1\) &        \(0\) &    \(1\)&        \(0\) &  (c) \\
 \(27\) &  \(-67128\) &  \(8,3,2797\) &  \((000000)\) & a.1 \(\uparrow\)       & \(\langle 5^4,7\rangle\)   & \(2\) &   \(2\) &        \(1\) &    \(1\)&        \(0\) &  (d) \\
 \(28\) &  \(-69619\) &   \(11,6329\) &  \((000000)\) & a.1                    & \(\langle 5^3,3\rangle\)   & \(2\) &   \(1\) &        \(0\) &    \(1\)&        \(0\) &  (c) \\
 \(29\) &  \(-70763\) &  \(7,11,919\) &  \((000000)\) & a.1 \(\uparrow\)       & \(\langle 5^4,7\rangle\)   & \(2\) &   \(1\) &        \(0\) &    \(1\)&        \(0\) &  (c) \\
 \(30\) &  \(-74019\) & \(3,11,2243\) &  \((000000)\) & a.1                    & \(\langle 5^3,3\rangle\)   & \(2\) &   \(1\) &        \(0\) &    \(1\)&        \(0\) &  (c) \\
 \(31\) &  \(-75103\) &   \(7,10729\) &  \((000000)\) & a.1                    & \(\langle 5^3,3\rangle\)   & \(2\) &   \(1\) &        \(0\) &    \(1\)&        \(0\) &  (c) \\
 \(32\) &  \(-75892\) &   \(4,18973\) &  \((100000)\) & a.2, fixed point       & \(\langle 5^4,8\rangle\)   & \(2\) &   \(1\) &        \(0\) &    \(1\)&        \(0\) &  (c) \\
 \(33\) &  \(-78747\) &   \(3,26249\) &  \((000000)\) & a.1                    & \(\langle 5^3,3\rangle\)   & \(2\) &   \(1\) &        \(0\) &    \(1\)&        \(0\) &  (c) \\
 \(34\) &  \(-83636\) &\(4,7,29,103\) &  \((000000)\) & a.1                    & \(\langle 5^3,3\rangle\)   & \(2\) &   \(1\) &        \(0\) &    \(1\)&        \(0\) &  (c) \\
 \(35\) &  \(-86404\) &   \(4,21601\) &  \((000000)\) & a.1                    & \(\langle 5^3,3\rangle\)   & \(2\) &   \(1\) &        \(0\) &    \(1\)&        \(0\) &  (c) \\
 \(36\) &  \(-91127\) &         prime &  \((000400)\) & a.2, fixed point       & \(\langle 5^4,8\rangle\)   & \(2\) &   \(1\) &        \(0\) &    \(1\)&        \(0\) &  (c) \\
 \(37\) &  \(-92219\) &         prime &  \((000000)\) & a.1                    & \(\langle 5^3,3\rangle\)   & \(2\) &   \(1\) &        \(0\) &    \(1\)&        \(0\) &  (c) \\
 \(38\) &  \(-99428\) & \(4,7,53,67\) &  \((003000)\) & a.2, fixed point       & \(\langle 5^4,8\rangle\)   & \(2\) &   \(1\) &        \(0\) &    \(1\)&        \(0\) &  (c) \\
\hline
\end{tabular}
\end{center}
\end{table}

%
%
%
%

\newpage

%
%
%
%

\renewcommand{\arraystretch}{1.0}

\begin{table}[ht]
\caption{{\sf The group \(\mathrm{G}_5^{(\infty)}{{M}}\) of \({M}=\mathbb{Q}\left((\zeta_5-\zeta_5^{-1})\sqrt{d}\right)\)  with \(-175000<d<-100000\)}}
\label{tbl:CycQrtFld4}
\begin{center}
\begin{tabular}{|r|r|c||c|l||c|c||c|c|c|c|c|}
\hline
 No.    & \multicolumn{2}{|c||}{Discriminant} & \multicolumn{2}{|c||}{Principalization} & \(\mathrm{G}_5^{(\infty)}{{M}}\) & \(\ell_5{M}\) & \multicolumn{5}{|c|}{Invariants} \\
        &       \(d\) & Factors       & \(\varkappa\) & Type                   &                            &       & \(r_1\) & \(\delta_1\) & \(r_2\) & \(\delta_2\) & Case \\
\hline
 \(39\) & \(-100708\) & \(4,17,1481\) &  \((000000)\) & a.1                    & \(\langle 5^3,3\rangle\)   & \(2\) &   \(1\) &        \(0\) &    \(1\)&        \(0\) &  (c) \\
 \(40\) & \(-101011\) &   \(83,1217\) &  \((000000)\) & a.1                    & \(\langle 5^3,3\rangle\)   & \(2\) &   \(1\) &        \(0\) &    \(1\)&        \(0\) &  (c) \\
 \(41\) & \(-101784\) &  \(8,3,4241\) &  \((003000)\) & a.2, fixed point       & \(\langle 5^4,8\rangle\)   & \(2\) &   \(1\) &        \(0\) &    \(1\)&        \(0\) &  (c) \\
 \(42\) & \(-104503\) &   \(7,14929\) &  \((000000)\) & a.1                    & \(\langle 5^3,3\rangle\)   & \(2\) &   \(2\) &        \(1\) &    \(1\)&        \(0\) &  (d) \\
 \(43\) & \(-105431\) & \(19,31,179\) &  \((000400)\) & a.2, fixed point       & \(\langle 5^4,8\rangle\)   & \(2\) &   \(1\) &        \(0\) &    \(1\)&        \(0\) &  (c) \\
 \(44\) & \(-105784\) &  \(8,7,1889\) &  \((000000)\) & a.1 \(\uparrow\)       & \(\langle 5^4,7\rangle\)   & \(2\) &   \(1\) &        \(0\) &    \(1\)&        \(0\) &  (c) \\
 \(45\) & \(-107791\) &         prime &  \((000000)\) & a.1                    & \(\langle 5^3,3\rangle\)   & \(2\) &   \(1\) &        \(0\) &    \(1\)&        \(0\) &  (c) \\
 \(46\) & \(-110479\) &         prime &  \((000000)\) & a.1                    & \(\langle 5^3,3\rangle\)   & \(2\) &   \(1\) &        \(0\) &    \(2\)&        \(1\) &  (e) \\
 \(47\) & \(-114303\) &  \(3,7,5443\) &  \((263415)\) & \(4\)-cycle            & \(\langle 5^5,11\rangle\)  & \(2\) &   \(1\) &        \(0\) &    \(1\)&        \(0\) &  (c) \\
 \(48\) & \(-114679\) &         prime &  \((000006)\) & a.2, fixed point       & \(\langle 5^4,8\rangle\)   & \(2\) &   \(1\) &        \(0\) &    \(1\)&        \(0\) &  (c) \\
 \(49\) & \(-115912\) &   \(8,14489\) &  \((000000)\) & abelian                & \(\langle 5^2,2\rangle\)   & \(1\) &   \(2\) &        \(0\) &    \(0\)&        \(0\) &  (a) \\
 \(50\) & \(-119191\) &         prime &  \((000000)\) & a.1                    & \(\langle 5^3,3\rangle\)   & \(2\) &   \(2\) &        \(1\) &    \(1\)&        \(0\) &  (d) \\
 \(51\) & \(-123028\) &   \(4,30757\) &  \((000000)\) & a.1                    & \(\langle 5^3,3\rangle\)   & \(2\) &   \(1\) &        \(0\) &    \(1\)&        \(0\) &  (c) \\
 \(52\) & \(-124099\) &   \(193,643\) &  \((000000)\) & a.1                    & \(\langle 5^3,3\rangle\)   & \(2\) &   \(1\) &        \(0\) &    \(1\)&        \(0\) &  (c) \\
 \(53\) & \(-125547\) &   \(3,41849\) &  \((003000)\) & a.2, fixed point       & \(\langle 5^4,8\rangle\)   & \(2\) &   \(1\) &        \(0\) &    \(1\)&        \(0\) &  (c) \\
 \(54\) & \(-127259\) & \(11,23,503\) &  \((000000)\) & a.1                    & \(\langle 5^3,3\rangle\)   & \(2\) &   \(1\) &        \(0\) &    \(1\)&        \(0\) &  (c) \\
 \(55\) & \(-127519\) &   \(7,18217\) &  \((000000)\) & a.1                    & \(\langle 5^3,3\rangle\)   & \(2\) &   \(1\) &        \(0\) &    \(1\)&        \(0\) &  (c) \\
 \(56\) & \(-133188\) &\(4,3,11,1009\)&  \((000000)\) & a.1                    & \(\langle 5^3,3\rangle\)   & \(2\) &   \(1\) &        \(0\) &    \(1\)&        \(0\) &  (c) \\
 \(57\) & \(-134392\) & \(8,107,157\) &  \((000000)\) & a.1                    & \(\langle 5^3,3\rangle\)   & \(2\) &   \(1\) &        \(0\) &    \(1\)&        \(0\) &  (c) \\
 \(58\) & \(-136311\) &  \(3,7,6491\) &  \((000000)\) & a.1                    & \(\langle 5^3,3\rangle\)   & \(2\) &   \(1\) &        \(0\) &    \(1\)&        \(0\) &  (c) \\
 \(59\) & \(-139703\) &         prime &  \((000000)\) & a.1                    & \(\langle 5^3,3\rangle\)   & \(2\) &   \(1\) &        \(0\) &    \(1\)&        \(0\) &  (c) \\
 \(60\) & \(-140232\) &  \(8,3,5843\) &  \((000000)\) & a.1                    & \(\langle 5^3,3\rangle\)   & \(2\) &   \(1\) &        \(0\) &    \(1\)&        \(0\) &  (c) \\
 \(61\) & \(-142904\) &   \(8,17863\) &  \((000000)\) & a.1                    & \(\langle 5^3,3\rangle\)   & \(2\) &   \(1\) &        \(0\) &    \(1\)&        \(0\) &  (c) \\
 \(62\) & \(-145007\) &         prime &  \((000000)\) & a.1                    & \(\langle 5^3,3\rangle\)   & \(2\) &   \(1\) &        \(0\) &    \(1\)&        \(0\) &  (c) \\
 \(63\) & \(-145668\) &\(4,3,61,199\) &  \((000000)\) & a.1                    & \(\langle 5^3,3\rangle\)   & \(2\) &   \(1\) &        \(0\) &    \(1\)&        \(0\) &  (c) \\
 \(64\) & \(-148004\) & \(4,163,227\) &  \((003000)\) & a.2, fixed point       & \(\langle 5^4,8\rangle\)   & \(2\) &   \(1\) &        \(0\) &    \(1\)&        \(0\) &  (c) \\
 \(65\) & \(-148507\) &   \(97,1531\) &  \((000000)\) & abelian                & \(\langle 5^2,2\rangle\)   & \(1\) &   \(2\) &        \(0\) &    \(0\)&        \(0\) &  (a) \\
 \(66\) & \(-151879\) & \(7,13,1669\) &  \((000000)\) & abelian                & \(\langle 5^2,2\rangle\)   & \(1\) &   \(2\) &        \(0\) &    \(0\)&        \(0\) &  (a) \\
 \(67\) & \(-154408\) &   \(8,19301\) &  \((000000)\) & abelian                & \(\langle 5^2,2\rangle\)   & \(1\) &   \(2\) &        \(0\) &    \(0\)&        \(0\) &  (a) \\
 \(68\) & \(-155603\) &   \(7,22229\) &  \((000000)\) & a.1                    & \(\langle 5^3,3\rangle\)   & \(2\) &   \(1\) &        \(0\) &    \(1\)&        \(0\) &  (c) \\
 \(69\) & \(-157028\) & \(4,37,1061\) &  \((003000)\) & a.2, fixed point       & \(\langle 5^4,8\rangle\)   & \(2\) &   \(1\) &        \(0\) &    \(1\)&        \(0\) &  (c) \\
 \(70\) & \(-157031\) &   \(7,22433\) &  \((000000)\) & a.1                    & \(\langle 5^3,3\rangle\)   & \(2\) &   \(1\) &        \(0\) &    \(1\)&        \(0\) &  (c) \\
 \(71\) & \(-159679\) & \(13,71,173\) &  \((000000)\) & a.1                    & \(\langle 5^3,3\rangle\)   & \(2\) &   \(1\) &        \(0\) &    \(1\)&        \(0\) &  (c) \\
 \(72\) & \(-160571\) &   \(211,761\) &  \((000000)\) & a.1                    & \(\langle 5^3,3\rangle\)   & \(2\) &   \(1\) &        \(0\) &    \(1\)&        \(0\) &  (c) \\
 \(73\) & \(-163427\) & \(11,83,179\) &  \((000050)\) & a.2, fixed point       & \(\langle 5^4,8\rangle\)   & \(2\) &   \(1\) &        \(0\) &    \(1\)&        \(0\) &  (c) \\
 \(74\) & \(-164116\) &  \(4,89,461\) &  \((000006)\) & a.2, fixed point       & \(\langle 5^4,8\rangle\)   & \(2\) &   \(1\) &        \(0\) &    \(1\)&        \(0\) &  (c) \\
 \(75\) & \(-165364\) &   \(4,41341\) &  \((000400)\) & a.2, fixed point       & \(\langle 5^4,8\rangle\)   & \(2\) &   \(1\) &        \(0\) &    \(1\)&        \(0\) &  (c) \\
 \(76\) & \(-169752\) &\(8,3,11,643\) &  \((000000)\) & a.1                    & \(\langle 5^3,3\rangle\)   & \(2\) &   \(1\) &        \(0\) &    \(1\)&        \(0\) &  (c) \\
\hline
\end{tabular}
\end{center}
\end{table}

%
%
%
%

\newpage

%
%
%
%

\renewcommand{\arraystretch}{1.0}

\begin{table}[ht]
\caption{{\sf The group \(\mathrm{G}_5^{(\infty)}{{M}}\) of \({M}=\mathbb{Q}\left((\zeta_5-\zeta_5^{-1})\sqrt{d}\right)\)  with \(-200000<d<-175000\)}}
\label{tbl:CycQrtFld5}
\begin{center}
\begin{tabular}{|r|r|c||c|l||c|c||c|c|c|c|c|}
\hline
 No.    & \multicolumn{2}{|c||}{Discriminant} & \multicolumn{2}{|c||}{Principalization} & \(\mathrm{G}_5^{(\infty)}{{M}}\) & \(\ell_5{M}\) & \multicolumn{5}{|c|}{Invariants} \\
        &       \(d\) & Factors       & \(\varkappa\) & Type                   &                            &       & \(r_1\) & \(\delta_1\) & \(r_2\) & \(\delta_2\) & Case \\
\hline
 \(77\) & \(-175076\) &\(4,11,23,173\)&  \((000000)\) & a.1                    & \(\langle 5^3,3\rangle\)   & \(2\) &   \(1\) &        \(0\) &    \(1\)&        \(0\) &  (c) \\
 \(78\) & \(-176459\) &         prime &  \((000000)\) & a.1                    & \(\langle 5^3,3\rangle\)   & \(2\) &   \(1\) &        \(0\) &    \(1\)&        \(0\) &  (c) \\
 \(79\) & \(-177428\) &   \(4,44357\) &  \((100000)\) & a.2, fixed point       & \(\langle 5^4,8\rangle\)   & \(2\) &   \(1\) &        \(0\) &    \(1\)&        \(0\) &  (c) \\
 \(80\) & \(-180583\) & \(13,29,479\) &  \((000000)\) & a.1                    & \(\langle 5^3,3\rangle\)   & \(2\) &   \(1\) &        \(0\) &    \(1\)&        \(0\) &  (c) \\
 \(81\) & \(-181847\) &   \(43,4229\) &  \((000000)\) & a.1                    & \(\langle 5^3,3\rangle\)   & \(2\) &   \(1\) &        \(0\) &    \(1\)&        \(0\) &  (c) \\
 \(82\) & \(-182968\) &   \(8,22871\) &  \((000050)\) & a.2, fixed point       & \(\langle 5^4,8\rangle\)   & \(2\) &   \(1\) &        \(0\) &    \(1\)&        \(0\) &  (c) \\
 \(83\) & \(-185883\) &   \(3,61961\) &  \((000000)\) & a.1                    & \(\langle 5^3,3\rangle\)   & \(2\) &   \(1\) &        \(0\) &    \(1\)&        \(0\) &  (c) \\
 \(84\) & \(-186187\) &         prime &  \((000400)\) & a.2, fixed point       & \(\langle 5^4,8\rangle\)   & \(2\) &   \(1\) &        \(0\) &    \(1\)&        \(0\) &  (c) \\
 \(85\) & \(-186271\) &         prime &  \((000050)\) & a.2, fixed point       & \(\langle 5^4,8\rangle\)   & \(2\) &   \(1\) &        \(0\) &    \(1\)&        \(0\) &  (c) \\
 \(86\) & \(-190387\) &         prime &  \((000000)\) & a.1                    & \(\langle 5^3,3\rangle\)   & \(2\) &   \(1\) &        \(0\) &    \(1\)&        \(0\) &  (c) \\
 \(87\) & \(-193483\) &  \(191,1013\) &  \((000000)\) & a.1                    & \(\langle 5^3,3\rangle\)   & \(2\) &   \(1\) &        \(0\) &    \(1\)&        \(0\) &  (c) \\
 \(88\) & \(-193571\) &   \(7,27653\) &  \((000000)\) & a.1                    & \(\langle 5^3,3\rangle\)   & \(2\) &   \(1\) &        \(0\) &    \(1\)&        \(0\) &  (c) \\
 \(89\) & \(-194487\) & \(3,241,269\) &  \((000000)\) & a.1  \(\uparrow\)      & \(\langle 5^4,7\rangle\)   & \(2\) &   \(1\) &        \(0\) &    \(1\)&        \(0\) &  (c) \\
 \(90\) & \(-196648\) &  \(8,47,523\) &  \((000000)\) & a.1                    & \(\langle 5^3,3\rangle\)   & \(2\) &   \(1\) &        \(0\) &    \(1\)&        \(0\) &  (c) \\
 \(91\) & \(-196707\) &\(3,7,17,19,29\)& \((000050)\) & a.2, fixed point       & \(\langle 5^4,8\rangle\)   & \(2\) &   \(1\) &        \(0\) &    \(1\)&        \(0\) &  (c) \\
 \(92\) & \(-197752\) & \(8,19,1301\) &  \((000000)\) & a.1                    & \(\langle 5^3,3\rangle\)   & \(2\) &   \(1\) &        \(0\) &    \(1\)&        \(0\) &  (c) \\
 \(93\) & \(-199947\) &\(3,11,73,83\) &  \((000400)\) & a.2, fixed point       & \(\langle 5^4,8\rangle\)   & \(2\) &   \(1\) &        \(0\) &    \(1\)&        \(0\) &  (c) \\
\hline
\end{tabular}
\end{center}
\end{table}\bigskip\bigskip\bigskip

%
%
%
%

%
%
%
%
%
%
%
%

\section{Acknowledgements}
\label{s:Acknowledgements}

\noindent
The third author gratefully acknowledges that his research was supported by the
Austrian Science Fund (FWF): P 26008-N25.

%
%
%
%
%
%
%
%


\end{document}